\documentclass[10pt]{amsart}
\usepackage{mathrsfs}
\usepackage{amsfonts}
\usepackage{amsmath}
\usepackage{amssymb}
\usepackage{amsthm,amscd}
\usepackage{tikz}
\usetikzlibrary{cd}
\newcommand{\Ha} {\mathbb {H}}
\newcommand{\D} {\mathbb {D}}
\newcommand{\R} {\mathbb {R} }

\newcommand{\N}{\mathbb {N}}
\newcommand{\C} {\mathbb {C}}
\newcommand{\Q} {\mathbb {Q}}
\newcommand{\T} {\mathcal {T}}
\newcommand{\dist} {\mathbf{d}}
\newcommand{\BD} {\mathcal {BD}}
\newcommand{\QD} {\mathcal {QD}}
\newcommand{\pt} {\partial}

\newcommand{\wh}{\widehat}
\newcommand{\id}{\operatorname{Id}}

\newcommand{\Thin} {\operatorname{Thin}}
\newcommand{\inj} {\operatorname{inj}}

\newcommand{\Stab} {\operatorname{Stab}}
\newcommand{\Aut} {\operatorname{Aut}}
\newcommand{\QMod} {\operatorname{QMod}}
\newcommand{\lcm} {\operatorname{lcm}}
\newcommand{\Img} {\operatorname{Im}}
\newcommand{\Real} {\operatorname{Re}}
\newcommand{\Teich}{{\mathcal {T}}}

\newcommand{\Belt} {\mathcal {BD}}

\newcommand{\Hol}  {\mathbf{\Phi}} 
\newcommand{\Fol}  {\mathbf{\Psi}} 
 
\newcommand{\imu}  {\mathbf{i}}
\newcommand{\Emb}  {\mathcal{E}} 
\newcommand{\Embo}  {\mathcal{F}} 

\newtheorem{corollary}{Corollary}[section]
\newtheorem{theorem}[corollary]{Theorem}
\newtheorem{lemma}[corollary]{Lemma}
\newtheorem{proposition}[corollary]{Proposition}
\newtheorem{claim}[corollary]{Claim}
\newtheorem{definition}[corollary]{Definition}

\theoremstyle{remark}
\newtheorem{remark}{Remark}

\begin{document}

\title[Carath\`eodory $\ne$ Teichm\"uller] {Carath\`eodory's  Metric on Teichm\"uller Spaces}
\author[Lin and  Markovic]{Yiran Lin \quad\quad Vladimir Markovi\'c}

\address{\newline YMSC   \newline Tsinghua University  \newline Beijing, China}

\today


\begin{abstract}  Let $S$ be an arbitrary Riemann surface whose Teichm\"uller space $\T(S)$  has dimension at least two.  A long standing problem (see \cite{F-M} and \cite{McMullen}) is to determine whether the Carath\'eodory metric $\dist_C$ agrees with the Teichm\"uller metric $\dist_\T$ on $\T(S)$. It was shown in \cite{Markovic} that $\dist_C\ne \dist_\T$ when $S$ is a closed surface of genus at least two. In this paper we study the general case, and prove that $\dist_C\ne \dist_\T$ on $\T(S)$ except possibly  on the following seven Teichm\"uller spaces: $\T^1_{0,0}$,  $\T^1_{0,1}$,  $\T^2_{0,0}$,  $\T^1_{0,2}$, $\T^2_{0,1}$,  $\T^3_{0,0}$, and $\T^3_{0,1}$.
\end{abstract}

\maketitle

\section{Introduction} Every complex Banach manifold $X$ naturally carries two pseudometrics, the \textsl{Carath\'eodory} pseudometric $\dist_C$, and the \textsl{Kobayashi} pseudometric $\dist_K$.  An important and well studied question is to decide whether $\dist_C\equiv \dist_K$ on a given complex manifold $X$. A particularly interesting case is when $X=\T(S)$ is the Teichm\"uller space of an arbitrary  Riemann surface $S$. Addressing this question  is the main subject of this paper.

\begin{remark} We observe that the Carath\'eodory metric $\dist_C$ should not be confused with the Carathéodory-Reiffen metric $\dist_{CR}$ which is obtained by integrating the infinitesimal form of $\dist_C$. The inequalities $\dist_C\le \dist_{CR}\le \dist_K$ hold on every $X$, and there are examples where each inequality is strict.
\end{remark}

Let $S$ be a  Riemann surface (all surfaces in this paper are assumed to orientable and connected). If $S$ is a $n$-punctured Riemann  sphere, for some $n\in \{0,1,2,3\}$, then the Teichm\"uller space $\T(S)$ consists of a single point. We exclude these cases from further discussions. Otherwise, the Teichm\"uller space $\T(S)$ is a complex Banach manifold (either of finite or infinite dimensions). By $\dist_\T$ we denote the Teichm\"uller metric on $\T(S)$.

A key result in Teichm\"uller theory  says that the Teichm\"uller and Kobayashi metrics agree on any Teichm\"uller space (see \cite{Royden}, \cite{E-K}, \cite{Gardiner}). On the other hand, it was shown in  \cite{Earle}  that  $\dist_C$ is a complete metric on $\T(S)$. After these initial results, the  question as to whether  $\dist_\T\equiv \dist_C$ on $\T(S)$ for an arbitrary Riemann surface $S$ came into focus. It is elementary to see that the two metrics agree on $\T(S)$ when $\dim(\T(S))=1$ (this holds when $S$ is a four punctured sphere, a torus, or a once punctured torus). Thus, the interesting case is when $\dim(\T(S))\ge 2$ (including the case  $\dim(\T(S))=\infty$).

\begin{remark} The  questions regarding the relationship between $\dist_\T$ and $\dist_C$ have been studied over the years. Important  results have been proved in \cite{Kra}, \cite{Krushkal}, \cite{E-H-H-M}, \cite{McMullen},   \cite{Markovic}, \cite{G-M}, \cite{D-M}, \cite{L-S}. The Carathéodory-Reiffen on Teichm\"uller spaces was extensively studied in \cite{Yeung}. 
\end{remark}

The Teichm\"uller space of closed Riemann surfaces of genus $g$, with $n$ punctures, and $b$ discs removed, is denoted by $\T_{g,n}^b$. In \cite{Markovic} it was  shown that $\dist_C\neq \dist_\T$ on $\T^0_{g,0}$, for every $g\ge 2$. In this paper, we attack the general case and prove that the same holds for  all but possibly seven  Teichm\"uller spaces (we remind the reader that there are uncountably many mutually non-isomorphic Teichm\"uller spaces \cite{Markovic-1}, \cite{E-M}).

\begin{theorem}\label{theorem: main theorem} Let $S$ be a  Riemann surface such that $\dim(\T(S))\ge 2$. Then $\dist_C\neq \dist_\T$ on $\T(S)$ unless  $\T(S)$ is one from the following list: $\T^1_{0,0}$,  $\T^1_{0,1}$,  $\T^2_{0,0}$,  $\T^1_{0,2}$, $\T^2_{0,1}$,  $\T^3_{0,0}$, and $\T^3_{0,1}$.
\end{theorem}

It remains an open problem to determine whether the same holds on these seven spaces.

\subsection{$\Sigma_{g,n}$-accommodating surfaces}  We let $\Sigma_{g,n}$ denote a topological surface of genus $g$ with $n$  holes.

\begin{definition} A Riemann surface $S$ is said to be $\Sigma_{g,n}$-accommodating if there exists an essential (topological) embedding $\Sigma_{g,n}\hookrightarrow S$. (Recall that an embedding is said to be essential if the induced map between the fundamental groups is injective.) 
\end{definition}

The first key idea in this paper is to show that if a Riemann surface $S$  is $\Sigma_{0,n}$-accommodating, and  $\dist_C\neq \dist_\T$ on $\T_{0,n}^0$, then $\dist_C\neq \dist_\T$ on $\T(S)$ as well. The proof of this fact rests on the following theorem.

\begin{theorem}\label{theorem: embedding Teichmuller space}
Let $S$ be a Riemann surface that is $\Sigma_{0,n}$-accommodating. Let $W\subset\T_{0,n}^0$ be a bounded open subset and let $\delta>0$ be any positive number. Then there exists a holomorphic map $H:W\rightarrow \T(S)$ such that
\begin{align}\label{eq1}
\dist_{\T}(X_1,X_2) \le \dist_{\T}(H(X_1),H(X_2))+\delta
\end{align}
for all $X_1,X_2\in W$.
\end{theorem}

As an immediate consequence, we get:
\begin{theorem}\label{theorem: accommodating case}
Let $S$ be a Riemann surface that is $\Sigma_{0,n}$-accommodating. Suppose that $\dist_C\neq \dist_\T$ on $\T_{0,n}^0$.
Then $\dist_C\neq \dist_\T$ on $\T(S)$.
\end{theorem}
\begin{proof} By assumption, we have $\dist_C\neq \dist_\T$ on $\T_{0,n}^0$. So there exists $X_1,X_2\in\T_{0,n}^0$ so that $\dist_C(X_1,X_2)<\dist_\T(X_1,X_2)$. Let $\delta>0$ be such that 
 $$
 \dist_C(X_1,X_2) < \dist_\T(X_1,X_2)-2\delta.
 $$ 
By Lemma \ref{lemma: basic lemma 3} we can chose a bounded open subset $W\subset\T_{0,n}^0$ containing 
$X_1,X_2$ such that 
$$
\dist_C^W(X_1,X_2)<\dist_C(X_1,X_2)+\delta.
$$ 
Here $\dist_C^W$ denotes the Carath\'eodory metric on the complex manifold $W$. Thus,
 \begin{equation}\label{eq2}
 \dist_C^W(X_1,X_2)<\dist_{\T}(X_1,X_2)-\delta.
 \end{equation}
By Theorem \ref{theorem: embedding Teichmuller space}, there exists $H:W\rightarrow \T(S)$ such that 
$$
\dist_{\T}(X_1,X_2)-\delta\le  \dist_{\T}(H(X_1),H(X_2))   .
$$ 
Combining with (\ref{eq2}) we get 
$$
\dist_C^W(X_1,X_2) < \dist_{\T}(H(X_1),H(X_2)).
$$ 
Since $H$ is holomorphic, from the Schwartz lemma de derive the inequality
$$
\dist_C(H(X_1),H(X_2))\leq \dist_C^W(X_1,X_2),
$$ 
which yields the inequality
$$
\dist_C(H(X_1),H(X_2)) <\dist_{\T}(H(X_1),H(X_2)). 
$$
So $\dist_C\neq \dist_\T$ on $\T(X)$, and we are done.
\end{proof}

\subsection{Model Teichm\"uller spaces} Theorem \ref{theorem: embedding Teichmuller space} suggests that we need to search for Teichm\"uller spaces  $\T^0_{0,n}$ on which $\dist_C\neq \dist_\T$. We think of such $\T^0_{0,n}$ as model spaces which can be ''almost" holomorphically embedded into $\T(S)$ if   $S$ is $\Sigma_{0,n}$-accommodating.  This is then used to show that  $\dist_C\neq \dist_\T$ on $\T(S)$. Such a model surface was found in \cite{Markovic}:
\begin{theorem}\label{theorem: 5-punctured sphere case} $\dist_C\neq \dist_\T$ on $\T^0_{0,5}$.
\end{theorem}
The  following lemma shows that ''most" surfaces are  $\Sigma_{0,5}$-accommodating.
\begin{lemma}\label{lemma-1} Suppose that a Riemann surface $S$ admits a pants decomposition with at least three pairs of pants.  Then $S$ is $\Sigma_{0,5}$-accommodating.
\end{lemma}
\begin{proof}
Since $S$ is connected, we can always pick three pairs of pants, say $P_1,P_2,P_3$, so that $P_1$ and $P_2$ are attached to each other along a common cuff $C_1$, and $P_2$ and $P_3$ are attached to each other along a common cuff $C_2$. Then $P_1\cup_{C_1} P_2\cup_{C_2} P_3$ is the desired essentially embedded $\Sigma_{0,5}$.
\end{proof}

The following lemma follows directly from Lemma \ref{lemma-1},  Theorem \ref{theorem: accommodating case}, and 
Theorem \ref{theorem: 5-punctured sphere case}.

\begin{lemma}\label{lemma-1-1} Suppose that a Riemann surface $S$ admits a pants decomposition with at least three pairs of pants.  Then $\dist_C\neq \dist_\T$ on $\T(S)$.
\end{lemma}

\subsection{Totally ramified coverings and a  3-punctured disc} 

Lemma \ref{lemma-1-1} implies that $\dist_C\neq \dist_\T$ on all but finitely many  Teichm\"uller spaces. However, we want to show that the same holds on all but seven  Teichm\"uller spaces. We employ a different method to show this.

\begin{definition}
Let $\pi:X\to Y$ be a branched covering between two surfaces $X$ and $Y$. We say that $\pi$ is fully ramified if no branch point has an unramified preimage.  
\end{definition}

In this subsection, we let $X$ and $Y$ denote two Riemann surface of finite topological type, and $\pi:X\to Y$ a finite degree holomorphic covering  branched over a finite (possibly empty) set $Q=\{q_1,\cdots,q_k\}\subset Y$. We have the induced holomorphic map 
$$
\pi^*:\T(Y\setminus Q)\to \T(X).
$$ 
The following is a theorem by Maclachlan–Harvey \cite{M-H}, and Winarski \cite{Winarski}.

\begin{theorem} If $\pi:X\to Y$ is fully ramified cover branched over $Q$, then $\pi^*:\T(Y\setminus Q)\to \T(X)$ is an isometric embedding with respect to the corresponding Teichm\"uller metrics on $\T(Y\setminus Q)$ and  $\T(X)$ respectively.
\end{theorem}

The following lemma is an immediate corollary of the previous theorem. 
A Riemann surface $X$ is said to be of type $(g,n,b)$ if it arises by removing $n$ punctures, and $b$ discs, from a closed Riemann surface of genus $g$. The  identity $\T(X)\equiv \T^b_{g,n}$ holds for any Riemann surface $X$ of type $(g,n,b)$.

\begin{lemma}\label{lemma-22} If $\pi:X\to Y$ is a fully ramified  cover, branched over a finite set $Q\subset Y$, and if $\dist_C\ne \dist_\T$ on $\T(Y\setminus Q)$, then $\dist_C\ne \dist_\T$ on $\T^b_{g,n}$ if $X$ is of type $(g,n,b)$.
\end{lemma}

The remaining cases of Theorem \ref{theorem: main theorem} will be proved by combining Lemma \ref{lemma-22} with the following theorem.

\begin{theorem}\label{theorem: 3-punctured disc case}
$\dist_C\neq \dist_\T$ on $\T^1_{0,3}$.
\end{theorem}

\subsection{Outline } The proof of our main result largely rests on Theorem \ref{theorem: embedding Teichmuller space} and Theorem \ref{theorem: 3-punctured disc case}. The argument behind the proof of Theorem \ref{theorem: embedding Teichmuller space}  is based on a novel construction in Teichm\"uller theory, and the first part of the paper is devoted to proving it
(Section \ref{section-thick} through Section \ref{section-posl}). The required holomorphic map $H:W\to \T(S)$ is constructed as a composition of  holomorphic maps $I:\T^n_{0,0}\to \T(S_0)$ and $J_\epsilon:W\to \T^n_{0,0}$ (the map $J_\epsilon$ is defined for $\epsilon>0$). We let $H_\epsilon=I\circ J_\epsilon$, and then show that  when  $\epsilon$ is sufficiently small we can take $H=H_\epsilon$. 

The map $I:\T^n_{0,0}\to \T(S_0)$  has the property that  for each $Y\in \T^n_{0,0}$ there exists a conformal embedding $\iota:Y \to S=I(Y)$ which is compatible with the marking $\Sigma_{0,n}\to Y$. To construct $I$ we first find $Y_0\in \T^n_{0,0}$ which allows a holomorphic embedding   $\iota:Y_0\to S_0$ realising the topological embedding $\iota:\Sigma_{0,n}\to S_0$ from the definition of $\Sigma_{0,n}$-accommodating surfaces. In turn, this allows us to construct the holomorphic map 
$\iota_*:\BD_1(Y_0)\to \BD_1(S_0)$ which ''embeds" the Beltrami differentials on $Y_0$  into the Beltrami differentials on $S_0$. The map $I$ is defined as a quotient of the map $\iota_*$, where we identify $\T^n_{0,0}$ with $\BD_1(Y_0)/\sim$, and $\T(S_0)$ with $\BD_1(S_0)/\sim$ (here $\sim$ is the Teichm\"uller equivalence relation among Beltrami differentials).

The second map is  $J_\epsilon:W\to \T^n_{0,0}$ has the following property. Suppose   $J_\epsilon(X)=Y\in \T^n_{0,0}$  for  some $X\in W$.  Then $Y$  is biholomorphic to  a subsurface of $X$, and the peripheral geodesics on $Y$ are shorter than $\epsilon$ (the geodesics are taken with respect to the complete hyperbolic metric on $Y$).

The second part is aimed at proving Theorem \ref{theorem: 3-punctured disc case}. As part of our analysis we observe that the theory developed in \cite{Markovic} generalises from analytically finite to topologically finite Riemann surfaces. Some additional justification is needed (see the appendix). We then show that $d_C\ne d_\T$ on the Teichm\"uller disc arising from a certain $L$-shaped 3-punctured disc. Namely,  the $L$-shaped 3-punctured disc $S(a,b,q)$ is the union of two copies of the $L$-shaped polygon 
$L(a,b,q)$ with all the edges identified except the edge $\overline{P_5P_1}$ (see Figure \ref{fig:Ls1} below).
The (2,0) form $dz^2$ on $L(a,b,q)$ gives rise to the quadratic differential $\psi(a,b,q)$ on $S(a,b,q)$. We show that $d_C\neq d_\T$ on the Teichm\"uller disc $\tau^{\psi(a,b,q)}$ generated by the quadratic differential $\psi(a,b,q)$.

Fix a base point $S_0=S(a_0,b_0,q_0)$. A holomorphic map $\Emb_0:\Ha^2\rightarrow\T(S_0)$ from the polyplane $\Ha^2$  to $\T(S_0)$ is constructed so that its restriction to the diagonal of $\Ha^2$ is equal to the Teichm\"uller disc $\tau^{\psi(a_0,b_0,q_0)}$. Assuming $d_C=d_\T$ on  $\tau^{\psi(a_0,b_0,q_0)}$ we construct  a holomorphic map $\Fol:\T(S_0)\rightarrow\Ha$ which restricts to the identity on the Teichm\"uller disc $\tau^{\psi(a_0,b_0,q_0)}$. Moreover, 
\begin{equation}\label{eq-proba}
(\Fol\circ\Emb_0)(\lambda_1,\lambda_2)=\alpha_1\lambda_1+\alpha_2\lambda_2,
\end{equation}
where $\lambda_1,\lambda_2\in \Ha$, and for suitable constants $\alpha_1,\alpha_2$.

A contradiction is  derived by showing that the holomorphic map $\Fol$ is not smooth at a point in $\T(S_0)$ which belongs to the boundary of the polyplane $\Emb_0(\Ha^2)\subset \T(S_0)$.
Namely, we consider  the smooth path $\sigma_1:[0,q_0)\to \T(S_0)$ given by $\sigma_1(t)=S(a,0,q_0-t)$, and show that $\Fol\circ \sigma_1$ is not smooth at $0$. The new moment here is that the path $\sigma_1$ does not belong to the polyplane   $\Emb_0(\Ha^2)$, so we can not use (\ref{eq-proba}) to show that  $\Fol\circ \sigma_1$ is not smooth at $0$. 

To remedy this issue we construct another path 
$\sigma_2:[0,q_0)\to \T(S_0)$ which does live in $\Emb_0(\Ha^2)$, and show that 
$$
(\Fol\circ \sigma_1)(t)-(\Fol\circ \sigma_2)(t)=O\left(\frac{t^2}{\log^2 t^{-1}}\right),
$$
which is enough to prove that  $\Fol\circ \sigma_1$ is not smooth at $0$.
\begin{remark} In the case when $S(a,b,q)$ is a 5-punctured sphere in $\T^0_{0,5}$, the path $\sigma_1$ is contained in the polyplane so it is not needed to construct the second path $\sigma_2$. This is the only substantial difference between the two cases. The only other difference concerns the computation of the derivative of $\Fol$. Since in our case $\T(S_0)$ is of infinite dimension more care is needed. This is done in the appendix.
\end{remark}

\subsection{Organization} In Section \ref{section-cara} we state some basic definitions and prove some basic claims about  the Carath\'eodory and Kobayashi pseudometrics.
In Section \ref{section-prvidokaz} we prove Theorem \ref{theorem: main theorem}. In
Section \ref{section-thick} through to Section \ref{section-posl}  we prove Theorem \ref{theorem: embedding Teichmuller space}. After this, the remainder of the paper is devoted to proving  
Theorem \ref{theorem: 3-punctured disc case}.

\section{The Carath\'eodory and Kobayashi pseudometrics}\label{section-cara}
The Carath\'eodory pseudometric $\dist_C$ is defined as:
\begin{equation}\label{eq: Caratheodory metric}
 \dist_C(p,q)=\sup\left\{d_\Ha(f(p),f(q)):f:X\rightarrow\Ha, f\text{ is holomorphic}\right\}.
\end{equation}
Here $\Ha$ denotes the upper half plane, and $d_\Ha$ the hyperbolic metric on $\Ha$. The Kobayashi pseudometric $\dist_K$ is defined as the largest possible pseudo metric on $X$ such that
\begin{equation}\label{eq: Kobayashi metric 1}
 \dist_K(p,q)\leq d_\Ha(\imu,z)
 \end{equation}
 for all holomorphic maps $f:\Ha\rightarrow X$ where $f(\imu)=p$, and $f(z)=q$ (here  $\imu\in\Ha$  denotes the imaginary unit).
\begin{remark} The Schwarz lemma implies $\dist_C(p,q)\leq \dist_K(p,q)$ for all pairs $p,q\in X$. 
\end{remark}

Next,  we prove a simple but general lemma asserting that in the definition of $\dist_C$ the supremum in (\ref{eq: Caratheodory metric}) is attained.
\begin{lemma}\label{lemma: basic lemma 1} For any $p,q\in X$, there exists a holomorphic map $\Phi:X\rightarrow\Ha$ such that $\dist_C(p,q)=d_{\Ha}(\Phi(p),\Phi(q))$. 
\end{lemma}
The main ingredient in the proof of Lemma \ref{lemma: basic lemma 1} is the notion of Banach limits which we now recall. Let $l^\infty(\N)$ be the Banach space of all bounded sequences of complex numbers with the supremum norm. Recall that as a simple consequence of the Hahn-Banach theorem, there exists a bounded linear functional $L:l^\infty(\N)\rightarrow\C$ with operator norm 1, with the properties:
\begin{enumerate}
\item $L((a_n)_{n\in\N})=a$ if $a_n\rightarrow a$, 
\vskip .1cm
\item $L((a_n)_{n\in\N})=L((a_{n+1})_{n\in\N})$.
\end{enumerate}
Such $L$ is called a Banach limit.
\begin{proof}    By definition there exists a sequence $\Phi_n:X\rightarrow\Ha$ such that 
\begin{itemize}
\item $\Phi_n(p)=\imu$,   
\vskip .3cm
\item $\lim_{n\to \infty} \Phi_n(q)=\lambda$, for some $\lambda\in \Ha$ such that $d_\Ha(\imu,\lambda)=\dist_C(p,q)$. 
\end{itemize}
Let $A$ be a M\"obius transform which takes $\Ha$ to $\D$, and such that $A(\imu)=0$. Then $(A\circ\Phi_n)_{n\in\N}$ is a holomorphic map from $X$ to the unit ball in $l^\infty(\N)$. Therefore $\Psi:=L\circ(A\circ\Phi_n)_{n\in\N}$ is a holomorphic map from $X$ to $\D$. It satisfies $\Psi(p)=0,\ \Psi(q)=A(\lambda)$. Then $\Phi:=A^{-1}\circ\Psi$ will satisfy the required property.
\end{proof}

If $W$ is an open subset of $X$, then $W$ is a complex manifold and it carries its own Carth\'eodory pseudometric. In this situation we use $d^W_C$ to denote the Carth\'eodory pseudometric on $W$ so we can distinguish it from the Carth\'eodory pseudometric $\dist_C$ on $X$. Note that  $\dist_C(p,q)\le \dist^W_C(p,q)$ when 
$p,q\in W$. By exactly the same argument as in the proof of the previous lemma, we get:
\begin{lemma}\label{lemma: basic lemma 3} Let $W_n\subset X$, $n\in \N$, be an increasing sequence of open sets whose union is equal to $X$. Then for every $p,q\in X$, we have \[\dist_C(p,q)=\lim_{n\rightarrow\infty}\dist_C^{W_n}(p,q).\]
\end{lemma}

\subsection{Holomorphic projections} We begin with the definition of holo-isometric discs.
\begin{definition} A  map $f:\Ha\to X$ is a holo-isometric disc if it is a  holomorphic isometry with respect to the hyperbolic metric on $\Ha$, and the Kobayashi metric on $X$, respectively. 
\end{definition}

\begin{remark}
Teichm\"uller spaces have ample holo-isometric  discs (examples being Teichm\"uller discs). In fact, every two points in a given Teichm\"uller space lie on the image of a holo-isometric disc.
\end{remark}

\begin{definition} Let $\Phi:X\rightarrow\Ha$ denote a holomorphic map. A  map $f:\Ha\to X$ is said to be  an extremal disc for $\Phi$ is $\Phi\circ f:\Ha\to \Ha$ is a conformal automorphism.
\end{definition}

\begin{lemma}\label{lemma-ext} Suppose $p_1,p_2\in X$ lie in the image of a holo-isometric disc  $f:\Ha\to X$.  If 
$\Phi:X\rightarrow\Ha$ is a holomorphic map such that $d_\T(p_1,p_2)=d_{\Ha}(\Phi(p_1),\Phi(p_2))$, then 
$f:\Ha\to X$ is an extremal disc for $\Phi$.
\end{lemma}
\begin{proof} 
Suppose  $p_i=f(z_i)$, $i=1,2$, where  $z_1,z_2\in \Ha$, and $z_1\ne z_2$. Since $f$ is a holo-isometric disc we conclude that $\Phi\circ f:\Ha\to \Ha$ is holomorphic, and that
$$
d_\Ha(z_1,z_2)=d_{\Ha}\big(\Phi(f(z_1)),\Phi(f(z_2))\big).
$$
It follows from  the Schwartz lemma that $\Phi\circ f$ is a conformal automorphism of $\Ha$.
\end{proof}

\section{Proof of Theorem \ref{theorem: main theorem}}\label{section-prvidokaz}

To prove the theorem we need to show that $\dist_C\neq \dist_\T$ on $\T(S)$ for every Riemann surface  $S$ satisfying the following two conditions:
\begin{enumerate}
\item  $\dim(\T(S))\ge 2$, 
\vskip .3cm
\item $S$ is not conformally equivalent to one of the following seven types of Riemann surfaces: a disc, an annulus, a 3-holed sphere, an once punctured disc, a twice punctured disc, an once punctured annulus, an once punctured 3-holed sphere.

\end{enumerate}

The proof of Theorem \ref{theorem: main theorem} is organised by the next two lemmas.

\begin{lemma}\label{lemma-3} If $S$ a Riemann surface satisfying the above conditions (1) and (2), then either $S$ has a pants decomposition with at least two pairs of pants, or $S$ is biholomorphic to a torus minus a disc.
\end{lemma}
\begin{proof} A Riemann surface whose Teichm\"uller space is non-trivial, and which does not allow a pants decomposition, is either  a disc, an annulus,  an once punctured disc, or a closed torus.
Since  the Teichm\"uller space of closed tori has dimension one, we conclude that if  $S$  satisfies  conditions $(1)$ and $(2)$, then it allows a pants decomposition. 

If $S$ consists of a single pair of pants then $S$ is either  a 3-holed sphere, a twice punctured disc, an once punctured annulus,  a once punctured torus, or a torus minus a disc.  Since  the Teichm\"uller space of once punctured tori has dimension one, we conclude that if  $S$  satisfies  conditions $(1)$ and $(2)$, then it either allow a pants decomposition consisting with at two least pairs of pants, or is biholomorphic to a torus minus a disc.
\end{proof}

\begin{lemma}\label{lemma-2} If $S$ is a Riemann surface satisfying the above conditions (1) and (2), and $S$ has a pants decomposition with exactly two pairs of pants, then one of the following holds:
\begin{itemize}
\item $S$ is of the type $(0,n,b)$, where $n+b=4$, and $b\ge1$,
\vskip .3cm
\item $S$ is of the type $(1,n,b)$, where $n+b=2$,
\vskip .3cm
\item $S$ is a closed surface of genus two.
\end{itemize}
\end{lemma}
\begin{proof}  We consider $S$ according to its genus. The largest genus such $S$ can have is two. Since $S$ consists of exactly two pairs of pants, it has to be a closed surface of genus two.  If $S$ has genus one, then it is obtained from a closed torus by removing $n$ punctures and $b$ discs, where  $n+b=2$. 

Finally, if $S$ has genus zero, then it is obtained from the Riemann sphere by removing $n$ punctures and $b$ discs, where  $n+b=4$. If $b=0$,  then $S$ is a four punctured sphere,  and it does not satisfy condition (1) above since in this case $\dim(S)=1$ . This concludes the proof.
\end{proof}

If $S$ has a pants decomposition with at least three pairs of pants then $\dist_C\neq \dist_\T$ on $\T(S)$  according to Lemma \ref{lemma-1-1}. The same holds if $S$  is a closed Riemann surface of genus $2$, or  a twice punctured torus,  as shown in \cite{Markovic}. Likewise, we derive the same conclusion from Theorem \ref{theorem: 3-punctured disc case} when $S$ is a  3-punctured disc.

Together with Lemma \ref{lemma-3} and Lemma \ref{lemma-2}, this reduces Theorem \ref{theorem: main theorem}  to the following statement which is proved in the next subsection.
 
\begin{lemma}\label{lemma-4} Suppose the triple $(g,n,b)$ satisfies one of the following conditions:
\begin{itemize}
\item $g=0$,  $n+b=4$, and $b\ge2$,
\vskip .3cm
\item $g=1$,  $n+b=2$, and $b\ge 1$,
\vskip .3cm
\item $g=1$,  $n=0$, and $b=1$.
\end{itemize}
Then  $\dist_C\neq \dist_\T$ on $\T^b_{g,n}$.  
\end{lemma}

\subsection{Proof of Lemma \ref{lemma-4}}  We begin with the following claim.

\begin{claim}\label{claim-1} Let $Y$ be a Riemann surface of type $(l,k,a)$, and $Q\subset Y$ a finite (possibly empty) set. Suppose that the set of punctures is divided into two subset $K_1$ and $K_2$, and that the set of holes is divided into subsets $A_1$ and $A_2$. Set $|K_i|=k_i$, $|A_i|=a_i$, and $|Q|=q$. If $q+a_1+k_1$ is a positive even integer, then there exists a Riemann surface $S$ of type $(g,n,b)$, and a fully ramified double cover $\pi:S\to Y$ branched over $Q$, such that 
$$
g=2k-1+\frac{1}{2}\big( q+k_1+a_1 \big),\quad n=k_1+2k_2,\quad b=a_1+2a_2.
$$
\end{claim}
\begin{proof}  Let $Y_1=Y\setminus Q$. There exists a double (unbranched) cover $\pi_1:S_1 \to Y_1$ such that 
each point from $Q$ has a single preimage in $S_1$,  and such that each  puncture from $K_1$, and each hole from $A_1$, lifts to a single puncture, and a single hole, respectively. Then, each  puncture from $K_2$, and each hole from $A_2$, has two preimages respectively.

Let $\pi:Y\to S$ be the induced fully ramified double cover. This cover is branched over $Q$. The genus $g$ of $S$ is calculated by the Riemann-Hurwitz formula. Moreover, $n=a_1+2a_2$, and $n=k_1+2k_2$. 
\end{proof}

To finish the proof of  Lemma \ref{lemma-4} we use the previous claim to find a fully ramified double cover  $\pi:S\to Y$, branched over $Q\subset Y$, in the situation when $\dist_C \ne \dist_\T$ on $\T(Y\setminus Q)$. By Lemma \ref{lemma-22} we then find that  $\dist_C \ne \dist_\T$ on $\T(S)$. The next three claims cover all cases from Lemma \ref{lemma-4}.

\begin{claim}\label{claim-2} $\dist_C \ne \dist_\T$ on $\T^2_{0,2}$, $\T^1_{1,0}$, and $\T^1_{1,1}$. 
\end{claim}
\begin{proof} We using the notation from Claim \ref{claim-1}. Set $l=0,k_1=0,k_2=1,a_1=0,a_2=1$, and  $q=2$. Then  $S$ is a Riemann surface of type $(0,2,2)$.  When $l=0,k_1=0,k_2=0,a_1=1,a_2=0, q=3$, then $(g,n,b)=(1,0,1)$. If $l=0,k_1=1,k_2=0,a_1=1,a_2=0, q=2$, then $(g,n,b)=(1,1,1)$. 
In all three cases $Y\setminus Q$ is of type $(0,3,1)$. The claim follows from Lemma \ref{lemma-22} and Theorem \ref{theorem: 3-punctured disc case}.
\end{proof}

\begin{claim} $\dist_C \ne \dist_\T$ on  $\T^4_{0,0}$.
\end{claim}
\begin{proof} Set  $l=0,k_1=0,k_2=0,a_1=0,a_2=2, q=2$. Then $S$ is of type $(g,n,b)=(0,0,4)$, while 
$Y\setminus Q$ is of type $(0,2,2)$. The claim follows from Lemma \ref{lemma-22} and the fact that  $\dist_C \ne \dist_\T$ on $\T^2_{0,2}$ which was proved in Claim \ref{claim-2}.
\end{proof}

\begin{claim} $\dist_C \ne \dist_\T$ on $\T^2_{1,0}$.  
\end{claim}
\begin{proof} Set $l=1,k_1=0,k_2=0,a_1=0,a_2=1$, and  $q=0$. Then  $S$ is a Riemann surface of type $(1,0,2)$, while    $Y\setminus Q$ is of type $(1,0,1)$.
The claim follows from Lemma \ref{lemma-22} and the fact that  $\dist_C \ne \dist_\T$ on $\T^1_{1,0}$ which was proved in Claim \ref{claim-2}.
\end{proof}

\section{Hyperbolic geometry of quasiconformal maps}\label{section-thick} In this section we recall some standard facts, and derive some basic consequences, relating quasiconformal maps and hyperbolic geometry.  Given a hyperbolic surface by $d_S$ we denote the complete hyperbolic metric on it.  The following is a well known corollary of the Mori's Theorem (see \cite{Ahlfors}).  

\begin{lemma}\label{lemma-mozda-1} Let $K\ge 1$. Then there exists a constant $D=D(K)>0$ such that if $f,g:S_1 \to S_2$ are two homotopic $K$-quasiconformal homeomorphisms between  two hyperbolic Riemann surfaces $S_1$ and $S_2$, then $d_{S_{2}}(f(z),g(z))\le D$, for every $z\in S_1$. 
\end{lemma}
We also have:

\begin{lemma}\label{lemma-mozda-2} Let $K\ge 1$. Then there exists a constant $L=L(K) \ge 1$ such that every $K$-quasiconformal homeomorphism  $f:S_1 \to S_2$ is  homotopic to a $L$-quasiconformal homeomorphism $g:S_1 \to S_2$ which is  
$L$-bilipschitz with respect to the metrics $d_{S_{1}}$ and $d_{S_{2}}$.
\end{lemma}
\begin{proof} Let $g$ be the Douady-Earle map \cite{D-E} which is homotopic to $f$. It is well known \cite{D-E} that there exists a constant $L\ge 1$ depending only on $K$ so that $g$ is both $L$-quasiconformal, and $L$-bilipschitz.
\end{proof}

\subsection{Injectivity radius and quasiconformal homeomorphisms}  For $z\in S$, we let $\inj_S(z)$ denote the injectivity radius of $z$ with respect to the metric $d_S$.
Recall that $\inj_S(z)$ is equal to half the length of shortest homotopically non-trivial closed curve containing $z$. 

We begin with the following elementary claim from hyperbolic geometry. The proof is left to the reader.

\begin{claim}\label{claim-mozda-3-0} Let $D>0$. Then there exists a  constant $C=C(D)\ge1$ such that for every hyperbolic or parabolic isometry $A:\Ha\to \Ha$, the inequality
\begin{equation}\label{eq-mozda-1-0}
d_\Ha(w,A(w))\le Cd_\Ha(z,A(z))
\end{equation}
holds for every $z,w\in \Ha$ when $d_\Ha(z,w)\le D$.
\end{claim}

\begin{lemma}\label{lemma-mozda-3} Let $D>0$. Then there exists a  constant $C=C(D)\ge 1$ such that 
$$
\inj_S(w)\le C\inj_S(z)
$$ 
holds for every $z,w\in S$ when $d_S(z,w)\le D$.
\end{lemma}
\begin{proof} Let $\Gamma$ be the Fuchsian group uniformising the surface $S$. Let $\wh{z},\wh{w}\in \Ha$ denote the lifts 
of $z$ and $w$ to the hyperbolic plane $\Ha$ such that $d_S(z,w)=d_\Ha(\wh{z},\wh{w})$.
Let $A\in \Gamma$ be such that $d_\Ha(\wh{z},A(\wh{z}))=2\inj_S(z)$.  Since $d_\Ha(\wh{z},\wh{w})\le D$,  from Claim \ref{claim-mozda-3-0} we conclude that 
$$
d_\Ha(\wh{w},A(\wh{w}))\le Cd_\Ha(\wh{z},A(\wh{z}))= 2C\inj_S(z),
$$
where $C=C(D)$ is the constant from the previous lemma. Since $2\inj_S(w)\le d_\Ha(\wh{w},A(\wh{w}))$,  the lemma follows.
\end{proof}

Given a $K$-quasiconformal homeomorphism $f:S_1\to S_2$ between two hyperbolic surfaces, we are interested in the relationship between $\inj_{S_{1}}(z)$ and $\inj_{S_{2}}(f(z))$ when $z\in S_1$.  

\begin{lemma}\label{lemma-mozda-4} Let $K\ge 1$. Then there exists a constant $C=C(K) \ge 1$ such that for every  $K$-quasiconformal homeomorphism  $f:S_1 \to S_2$, the inequality
\begin{equation}\label{eq-mozda-1}
\inj_{S_{2}}(f(z))\le C \inj_{S_{1}}(z),
\end{equation}
holds for every $z\in S_1$.
\end{lemma}
\begin{proof} Let $L=L(K)$ be the constant  from Lemma \ref{lemma-mozda-2},  let $D=D(\max\{K,L\})$ be the constant from Lemma \ref{lemma-mozda-1}, and let $C_1=C_1(D)$ the constant from Lemma  \ref{lemma-mozda-3}.  Let $g:S_1 \to S_2$ be a  $L$-quasiconformal, and $L$-bilipschitz map,  homotopic to $f$. 

Since $g$ is $L$-bilipschitz, it follows that 
\begin{equation}\label{eq-moze-1}
\inj_{S_{2}}(g(z))\le L \inj_{S_{1}}(z).
\end{equation} 
Since  both $f$ and $g$ are $(\max\{K,L\})$-quasiconformal, we  have that 
$$
d_{S_{2}}(f(z),g(z)) \le D.
$$ 
Thus, from Lemma  \ref{lemma-mozda-3} we derive the inequality 
$$
\inj_{S_{2}}(f(z))\le C_1 \inj_{S_{2}}(g(z)). 
$$
We now apply the estimate (\ref{eq-moze-1}) to the right hand side of the previous inequality to obtain
$$
\inj_{S_{2}}(f(z))\le  LC_1  \inj_{S_{1}}(z). 
$$
Letting $C=LC_1$ proves the lemma.
\end{proof}

\subsection{Think-Thin decompositions} By $\Thin_r(S)$ we denote the $r$-thin part of $S$, that is, 
$$
\Thin_r(S)=\{  z\in S \, : \,   \inj_S(z) < r     \}.
$$
The following lemma is a direct corollary of Lemma \ref{lemma-mozda-4}.
\begin{lemma}\label{lemma-mozda-5} Let $K\ge 1$. Then there exists a constant $C=C(K) \ge 1$ such that for every  $K$-quasiconformal homeomorphism  $f:S_1 \to S_2$, the inclusion 
\begin{equation}\label{eq-mozda-2}
f\big(\Thin_{r}(S_1)\big)\subset  \Thin_{Cr}(S_2).
\end{equation}
\end{lemma}

We record the following claim which follows from the definition of the injectivity radius. 

\begin{claim}\label{claim-tvrd-00} Let $\gamma \subset S$ be a homotopically non-trivial  closed curve. 
Then $\gamma \subset \Thin_r(S)$ whenever $\l_S(\gamma)<2r$. The component of $\Thin_r(S)$ containing $\gamma$ is denoted by $\Thin_r(S)[\gamma]$.
\end{claim}

We finish this section by recalling the definition of the Margulis constant. 
Namely, there exists a constant $r_M$ such that when $0<r<r_M$, each connected component of  $\Thin_r(S)$ is either a tube around an embedded geodesic of length at most $2r_M$,  or is a  product neighbourhood of  a cusp. We record another elementary claim.

\begin{claim}\label{claim-tvrd-11}  Let $\gamma\subset S$ be a homotopically non-trivial  closed curve. 
If  $\l_S(\gamma)<2r < 2r_M$, then $\Thin_r(S)[\gamma]$ is an annulus. If $\beta$ is a core curve of this annulus then  $\gamma$ is homotopic to a power of $\beta$.
\end{claim}

\subsection{Essential conformal embeddings}\label{section-embed} Let $Y\in \T^n_{0,0}$, and fix a hyperbolic surface $S$.
Throughout this subsection, we consider an essential conformal embedding  $\iota:Y\to S$, and relate the injectivity radii $\inj_Y(z)$ and $\inj_S(\iota(z))$. 

Recall that  $d_Y$ denotes the (complete) hyperbolic metric on $Y$, and $d_S$ the one on $S$.  
Let $\alpha_i \subset Y$, $1\le i\le n$,  be the closed geodesics (with respect to $d_Y$) homotopic to the  ends of $Y$. We refer to $\{\alpha_i\}_{1\le i\le n}$,  as the peripheral geodesics on $Y$.

We let $\beta_i=\iota(\alpha_i)$. The embedding $\iota$ is holomorphic so by the Schwartz lemma we have
\begin{equation}\label{eq-tvrd-1-01}
d_S(\iota(z),\iota(w))\le d_Y(z,w),
\end{equation}
for every two points $z,w\in Y$. In particular, we 
derive the inequality 
\begin{equation}\label{eq-tvrd-1}
\l_S\big( \beta_i \big)\le \l_Y(\alpha_i).
\end{equation}
Define the  union
$$
\Thin_r(S)[\beta]\overset{def}{=}\bigcup _{i=1}^n \Thin_r(S)[\beta_i].
$$
If $r<r_M$ then two annuli $\Thin_r(S)[\beta_i]$ and $\Thin_r(S)[\beta_j]$ are either disjoint or equal to each other. Thus, in this case the set $S\setminus \Thin_r(S)[\beta_i]$ is collection of disjoint essential subsurfaces  of $S$, exactly one of which is homotopy equivalent to $\iota(Y)$.

\begin{lemma}\label{lemma-tvrd} For every $D>0$ there exists  $C=C(D)\ge 1$ with the following properties. 
Let $\gamma\subset Y$ be a  homotopically non-trivial closed curve which is not  homotopic to a power of one of the peripheral geodesics of $Y$. If $\l_Y(\gamma)\le 2D$, and if $\iota(\gamma)$ intersects  $\Thin_r(S)[\beta]$, then
\begin{equation}\label{eq-tvrd-2}
r\ge \frac{r_M}{C}.
\end{equation}
\end{lemma}
\begin{proof} Let $C=C(D)\ge 1$ be the constant from Lemma \ref{lemma-mozda-3}. From (\ref{eq-tvrd-1-01}), and since $\l_Y(\gamma)\le 2D$, it follows that the distance  between any two points on 
$\iota(\gamma)$ is at most $D$. Therefore,  if $\iota(\gamma)\cap \Thin_r(\beta)\ne \emptyset$, it follows that  
$$
\iota(\gamma)\subset \Thin_{Cr}(\beta).
$$ 

If $Cr<r_M$, then $\iota(\gamma)$ belongs to one of the disjoint collection of annuli homotopic to the curves 
$\beta_i$, $1\le i \le n$. But then by Claim \ref{claim-tvrd-11}, $\iota(\gamma)$ is homotopic to a power of one of 
$\beta_i$, $1\le i \le n$. Since $\iota$ is an essential embedding, it follows that $\gamma$ is homotopic to a power of one of the peripheral geodesics on $Y$ which is not possible. This shows that $Cr\ge r_M$, and we are done.
\end{proof}

\section{Teichm\"uller spaces and Beltrami differentials} 

In this section we recall some standard  definitions and results from \cite{Ahlfors}.

\subsection{Teichm\"uller spaces}
Fix a Riemann surface $Z$. Recall  the definition of the Teichm\"uller space $\T(Z)$. Consider the pairs $(Z_1,f_1)$ consisting of a Riemann surface $Z_1$ and a quasiconformal homeomorphism $f_1:Z\to Z_1$. Two such pairs $(Z_1,f_1)$ and $(Z_2,f_2)$ are said to be equivalent if the quasiconformal homeomorphism $f_2 \circ f_1^{-1}:Z_1\to Z_2$ is homotopic (relative  the boundary) to a conformal map. 

\begin{remark} The homotopy here is required to be covered by an equivariant homotopy $h_t:\Ha\rightarrow\Ha$ of the universal covering such that $h_t$ extends continuously up to the boundary of $\Ha$ and is equal to a conformal automorphism on $\pt{\Ha}$ (this requirement is automatic when $Z$ is of finite analytic type).
\end{remark}

The space of equivalence classes $[(Z_1,f_1)]$ is  called the Teichm\"uller space. The homotopy class of the homeomorphism $f_1$ is called the marking. To simplify the notation when we often write $Z_1\in \T(Z)$. However, we always assume that such $Z_1$ is marked by the homotopy class of a quasiconformal homeomorphism $f_1:Z\to Z_1$ .

\subsection{Beltrami differentials} 
A Beltrami differential on  $Z$ is a $(-1,1)$ complex form. The Banach  space of Beltrami differentials of finite supremum norm  is denoted by  $\BD(Z)$. The unit ball in $\BD(Z)$ is denoted by   $\BD_1(Z)$. Given $\mu\in \BD_1(Z)$, we let $f_\mu:Z\to Z_\mu$ denote the quasiconformal homeomorphism with the Beltrami dilatation equal to $\mu$, that is,
$$
\mu=\frac{\overline{\pt}{f_\mu}}{\pt{f_\mu}}.
$$   

Two Beltrami differentials $\mu,\nu\in \BD_1(Z)$ are equivalent $\mu \sim \nu$ if the map $f_\nu\circ f_\mu^{-1}:Z_\mu\to Z_\nu$ is homotopic to a conformal map. This can be rephrased as follows: there exists a quasiconformal homeomorphism $g:Z\to Z$, which is homotopic to the identity map relative the boundary, such that the Beltrami dilatation of the map $f_\mu\circ g$ is equal to $\nu$. 

Furthermore, we have the biholomorphic identification
\begin{equation}\label{eq-iden}
\{[\mu]: \mu \in \BD_1(Z)\} \equiv \T(Z),
\end{equation}
where $[\mu]$ denotes the equivalence class of $\mu$.
In other words, $\T(Z)$ has the structure of a complex Banach manifold which is compatible with the one on $\mathcal{BD}(Z)$.

\subsection{The tangent space of $\T(Z)$}\label{subsection-well}

By $\QD(Z)$ we denote the Banach space of holomorphic quadratic differentials on $Z$  with respect to the $L^1$ norm. 
Every $\mu \in \BD(Z)$ induces the linear functional $\lambda_\mu:\QD(Z)\to \C$ by
$$
\lambda_\mu(\varphi)=\int\limits_{Z} \mu \varphi, \quad \quad\quad \forall \varphi\in \QD(Z).
$$

Define the Banach subspace  $\BD_0(Z)\subset \BD(Z)$ by
$$
\BD_0(Z)=\{\mu\in \BD(Z): \lambda_\mu(\varphi)=0,\,\,\,\forall \varphi\in \QD(Z) \}.
$$
Note that  the quotient Banach space $\BD(Z)/\BD_0(Z)$ is isometric to the  dual of the Banach space 
$\QD(Z)$. For $\mu\in \BD(Z)$, we let $[\mu]^*$ denote its its image in the quotient $\BD(Z)/\BD_0(Z)$.

Let $T_Z\T(Z)$ denote the tangent space at  $Z$. We equip $T_Z\T(Z)$ with the Finsler norm induced by the Teichm\"uller metric. This turns $T_Z\T(Z)$ into a Banach space.

Analogously to (\ref{eq-iden}), we have  that  
the Banach space $T_Z\T(Z)$ is isometric to the quotient Banach space $\BD(Z)/\BD_0(Z)$. 
Therefore,  $T_Z\T(Z)$ is isometric to the dual of $\QD(Z)$.

\section{The map $I:\T^n_{0,0}\to \T(S_0)$}

Throughout this section we assume that $S_0$ is  $\Sigma_{0,n}$-accommodating Riemann surface, and we fix an essential (topological)  embedding   $\iota_0:\Sigma_{0,n} \hookrightarrow S_0$.  The goal of this section is to prove the following theorem.

\begin{theorem}\label{thm-1}  There exists a holomorphic map $I:\T^n_{0,0}\to \T(S_0)$ such that for every $Y\in \T^n_{0,0}$ there exists a conformal embedding $\iota:Y \to S$ homotopic to $\iota_0$ (with respect to the markings $\Sigma_{0,n}\to Y$ and $S_0\to S$). Here $S=I(Y)$. 
\end{theorem}

\subsection{The map $I:\T^n_{0,0}\to \T(S_0)$} 

By homotoping the embedding $\iota_0:\Sigma_{0,n}\hookrightarrow S_0$, we may assume that the image $\iota_0(\Sigma_{0,n}) \subset S_0$ is an embedded  subsurface whose boundary consists of $n$ disjoint smooth curves. This embedding equips $\Sigma_{0,n}$ with the  structure of a Riemann surface which we denote by $Y_0$, and we have  the conformal embedding $\iota_0:Y_0 \hookrightarrow S_0$.  Clearly, $Y_0$ is a Riemann surface of type $(0,0,n)$ whose boundary  consists of $n$ smooth curves.
   
The inclusion $\iota_0:Y_0\to S_0$ induces the ''inclusion" map $\iota^*:\BD_1(Y_0)\to \BD_1(S_0)$ as follows. Let $\mu\in \BD_1(Y_0)$. Then we let $\iota^*(\mu)$ equal to $\mu$ on $Y_0\subset S_0$, and equal to zero on $S_0\setminus Y_0$. The map $\iota^*$ is clearly  holomorphic.

\begin{lemma}\label{lemma-munu} Let $\mu,\nu \in \BD_1(Y_0)$, and suppose $\mu\sim \nu$ on $Y_0$. Then $\iota^*(\mu)\sim \iota^*(\nu)$ on $S_0$. 
\end{lemma}
\begin{proof}  Set $\mu_1=\iota^*(\mu)$, and  $\nu_1=\iota^*(\nu)$. Since $\mu\sim \nu$, there exists a quasiconformal homeomorphism $g:Y_0\to Y_0$, which is homotopic to the identity on $Y_0$, and equal to the identity on $\pt{Y_0}$, such that the Beltrami dilatation of $f_\mu\circ g$ is equal to $\nu$.

Let $g_1:S_0\to S_0$ be the map which is equal to $g$ on $Y_0$, and equal to the identity map on $S_0\setminus Y_0$. By construction,  $g_1$ is a self-homeomorphism of $S_0$ which is homotopic to the identity. Therefore, $g_1$ is quasiconformal when restricted on $Y_0$, and when restricted on $S_0\setminus Y_0$. Since $\pt{Y_0}$ are smooth curves we conclude that $g_1$ is quasiconformal.  Moreover, the Beltrami dilatation of the map $f_{\mu_{1}}\circ g_1$ is equal to $\nu_1$. Thus, $\mu_1\sim \nu_1$.
\end{proof}

\subsection{Proof of Theorem \ref{thm-1}}
Lemma \ref{lemma-munu} implies that the holomorphic map $\iota^*:\BD_1(Y_0)\to \BD_1(S_0)$ respects the equivalence relation $\sim$. This induces the required holomorphic map  $I:\T(Y_0)\to \T(S_0)$ (note that $\T(Y_0)\equiv \T^n_{0,0}$). 
It remains to show that for every $Y\in \T^n_{0,0}$ we construct a conformal embedding $\iota:Y \to S$ homotopic to $\iota_0$, where $S$ denotes the marked Riemann surface underlying the point $I(Y)\in \T(S_0)$).

Let $\mu\in \BD_1(Y_0)$ be such that $[\mu]=Y$, and let $f_\mu:Y_0\to Y$ be the resulting quasiconformal homeomorphism.
Set  $\mu_1=\iota^*(\mu)$, and let $f_{\mu_{1}}:S_0\to S$ be the resulting quasiconformal homeomorphism. 
Define $\iota=f_{\mu_{1}}\circ f^{-1}_\mu$. Then $\iota:Y\to S$ is the required conformal embedding.

\section{The map $J_\epsilon:W\to \T^n_{0,0}$} The goal of this section is to construct the map $J_\epsilon:W\to \T^n_{0,0}$, and prove Theorem \ref{thm-2}.  
As in Section \ref{section-embed}, for a given   $Y\in \T^n_{0,0}$ we let $d_Y$ denote the (complete) hyperbolic metric on $Y$. By $\alpha_i \subset Y$, $1\le i\le n$,  we denote the closed geodesics (with respect to $d_Y$) homotopic to the  ends of $Y$, which we call the peripheral geodesics on $Y$.

Let $W \subset \T^0_{0,n}$ denote a bounded domain.
In this section we prove a $n$-punctured sphere $X\in W$ contains an embedded and homotopically equivalent $n$-holed sphere $Y\subset X$ which varies  holomorphically with $X$. Moreover, we can arrange that the peripheral geodesics of $Y$ are as short as we like.

\begin{definition} Let $\epsilon>0$, and   $X\in \T^0_{0,n}$. We say that  $Y\subset X$ is an 
$\epsilon$-subsurface if
\begin{enumerate}
\item $Y$ is homotopy equivalent to $X$, and  of conformal type $(0,0,n)$,
\vskip .1cm
\item $0<\l_Y(\alpha_i)\le \epsilon$, for every $1\le i\le n$. 
\end{enumerate}
\end{definition}
\begin{theorem}\label{thm-2}  Let $W \subset \T^0_{0,n}$ be a bounded domain, and let $\epsilon>0$.  There  exists a holomorphic map $J_\epsilon:W\to \T^n_{0,0}$ such that $J_\epsilon(X)=Y$ is an $\epsilon$-subsurface of $X$.
\end{theorem}

Since every bounded domain $W\subset \T^0_{0,n}$ is contained in a ball (with respect to the Teichm\"uller metric), it suffices to prove Theorem \ref{thm-2} assuming $W$ is a ball. In the remainder of this section we assume $W$ is a ball centred at $X_0\in \T^0_{0,n}$. We also fix $\epsilon>0$.

\subsection{Holomorphically varying neighbourhoods of the punctures}

There are holomorphic maps 
$p_i:\T^0_{0,n}\to \wh{\C}$ such that  
\begin{equation}\label{eq-X}
X\equiv \wh{\C}\setminus \{p_1(X),\cdots p_n(X)\},
\end{equation}
and that $p_i(X)\ne p_j(X)$ when  $i\ne j$ for every  $X\in \T^0_{0,n}$.

\begin{lemma}\label{lemma-lem-1} There are holomorphic maps $A_i:\T^0_{0,n}\to \text{Aut}(\wh{\C})$, $1\le i\le n$, such that 
the M\"obius transformation $A_i(X)$ maps the point $p_i(X_0)$ to the point  $p_i(X)$.
\end{lemma}
\begin{proof} Let $A(X)\in \text{Aut}(\wh{\C})$ be the M\"obius transformation mapping the points $p_1(X_0), p_2(X_0), p_3(X_0),$ to the points $p_1(X), p_2(X), p_3(X),$ respectively, and set $A_1(X)=A_2(X)=A_3(X)=A(X)$. 
Furthermore, for each $i\ge 4$, there exists a unique  M\"obius transformation $A_i(X)\in \text{Aut}(\wh{\C})$  which maps the points  $p_1(X_0), p_2(X_0), p_i(X_0),$ to the points $p_1(X), p_2(X), p_i(X),$ respectively. Since the points $p_i(X)$ depend holomorphically in $X$ so do the transformations $A_i(X)$.
\end{proof}

Denote by $D_i(X_0,\delta)\subset \wh{\C}$ the  ball of radius $\delta$  (with respect to the spherical metric on $\wh{\C}$)  centred at $p_i(X_0)$. Define 
$$
D_i(X,\delta)=A_i(X)\left(D_i(X_0,\delta)\right).
$$
We think of $D_i(X,\delta)$ as a holomorphically varying neighbourhood of the puncture $p_i(X)$.

\begin{claim}\label{claim-100} There exists $\delta>0$ small enough so that
\begin{itemize}
\item   for every $X\in W$ the sets  $D_i(X,\delta)$, $1\le i \le n$, are mutually disjoint, 
\vskip .1cm
\item the surface $Y_X(\delta)\subset X$  given by
$$
Y_X(\delta)=\wh{\C}\setminus \bigcup_{i=1}^{n} D_i(X,\delta)
$$
is an  $\epsilon$-subsurface of $X$.
\end{itemize}
\end{claim}
\begin{proof} The proof follows immediately from the facts that $W$ is compactly contained in $\T(X_0)$, and that 
$\epsilon$ is  fixed.
\end{proof}
Choose $\delta$ satisfying the conclusion of the previous claim. Let $D_i(X)=D_i(X,\delta)$, and $Y_X=Y_X(\delta)$. Then $Y_X$ is $\epsilon$-subsurface of $X\in W$.

\subsection{Proof of Theorem \ref{thm-2}} 
We  define the map $J_\epsilon:W \to \T^n_{0,0}$ as follows.

\begin{definition} Let $X\in \T^n_{0,0}$. We say that  $f\in F_X$ if
\begin{enumerate}
\item $f:X_0\to X$ is a quasiconformal homeomorphism such that  $X=[(X,f)] \in W$, 
\vskip .1cm
\item $f$ agrees with the M\"obius transformation $A_i(X)$ when restricted to the disc $D_i(X_0)$, for each $1\le i \le n$.
\end{enumerate}
\end{definition}
We let $\wh{f}:Y_{X_{0}}\to Y_X$ denote the corresponding restriction of $f\in F_X$.  Then   $[(Y_X,\wh{f})]$ is a well defined point in  $\T^n_{0,0}$. Although   $[(Y_X,\wh{f_1})]$ and  $[(Y_X,\wh{f_2})]$
may not be the same point in  $\T^n_{0,0}$ for some $f_1,f_2\in F_X$, they can only differ by a Dehn multi-twist about the peripheral geodesics $\{\gamma_i\}_{1\le i\le n}$. Thus, the set 
$\{[(Y_X,\wh{f})]\,:\, f\in F_X\}$ is a discrete subset of  $\T^n_{0,0}$.

Now, since $W$ is simply connected, and since $\{[(Y_X,\wh{f})]\,:\, f\in F_X\}\subset \T^n_{0,0}$ is discrete,
there exists  a unique continuous section $J_\epsilon:W\to \T^n_{0,0}$ such that 
$$
J_\epsilon(X)\in \{[(Y_X,\wh{f})]\,:\, f\in F_X\},
$$
and
$$
J_\epsilon(X_0)=[(Y_{X_{0}},\wh{\id})].
$$
We have already observed that $Y_X$ is an $\epsilon$-subsurface of $X$. It is holomorphic since $A_i(X)$ depends holomorphically on $X$ (this is a standard argument using Slodkowski's theorem \cite{Slodkowski} about extending holomorphic motions).

\section{Proof of Theorem \ref{theorem: embedding Teichmuller space}}\label{section-posl} Let $W_1\subset \T^0_{0,n}$ be an open ball containing the closed ball $\overline{W}$, and let $J_\epsilon:W_1\to \T^n_{0,0}$ be the corresponding map 
from Theorem \ref{thm-2}. In particular, the map $J_\epsilon$ is well defined on some neighbourhood of $\overline{W}$. Set $H_\epsilon=I\circ J_\epsilon$. Then $H_\epsilon:W_1\rightarrow \T(S)$ is a holomorphic map.  

In the remainder of the proof  we show that for a given $\delta>0$ there exists $\epsilon>0$ such that 
\begin{equation}\label{eq711}
\dist_{\T}(X_1,X_2) \le \dist_{\T}(H_\epsilon(X_1),H_\epsilon(X_2))+\delta
\end{equation}
for all pairs $(X_1,X_2) \in \overline{W} \times \overline{W}$. Since  $ \overline{W} \times \overline{W}$ is compact,  it suffices to prove that for a fixed pair $(X_1,X_2) \in \overline{W} \times \overline{W}$ the inequality (\ref{eq711}) holds when $\epsilon$ is small enough.

The proof is by contradiction.  That is, we suppose  that there exist points $X_1,X_2 \in \overline{W}$, and 
a sequence $\epsilon_k\to 0$, when $k\to \infty$,  such that 
\begin{equation}\label{eq-fin}
d_{\T}(S^1_k,S^2_k)+\delta < d_{\T}(X_1,X_2),\quad\quad \text{for every $k$},
\end{equation}
where $S^j_k=H_{\epsilon_{k}}(X_j)$, $j=1,2$. Below we derive a contradiction from (\ref{eq-fin}).

We shall prove the following lemma.
\begin{lemma}\label{lemma-fin}
\begin{equation}\label{eq-fin-2}
d_{\T}(X_1,X_2)\le \limsup\limits_{k\to \infty} d_{\T}(S^1_k,S^2_k).
\end{equation}
\end{lemma}
Before that, we complete the proof of  Theorem \ref{theorem: embedding Teichmuller space} assuming  Lemma \ref{lemma-fin}.
By taking $\limsup$ of the left hand side in  (\ref{eq-fin}), we get
$$
\limsup\limits_{k\to \infty}d_{\T}(S^1_k,S^2_k)+\delta \le d_{\T}(X_1,X_2).
$$
Replacing (\ref{eq-fin-2}) in the previous inequality, we obtain
$$
d_{\T}(X_1,X_2)+\delta \le d_{\T}(X_1,X_2).
$$
This contradiction completes the proof. The remainder of this section is devoted to proving  Lemma \ref{lemma-fin}.

\subsection{The subsurface $Z_k(r)\subset S_k$} Fix $X\in \overline{W}$, and let $H_{\epsilon_k}(X)=S_k$. By $Y_k \subset X$ we denote the  $\epsilon_k$-subsurfaces   such that $J_{\epsilon_{k}}(X)=Y_k$, and by  
$\iota_k:Y_k\to S_k$ the conformal embeddings from Theorem \ref{thm-1}. As above, by $\alpha^k_i \subset Y_k$, $1\le i\le n$, we denote the peripheral geodesics of $Y_k$, and 
set $\beta^k_i=\iota_k(\alpha^k_i)$.

Assume $\epsilon_k<2r$. Then $\forall \,\,  i\in\{1,...,n\}$, we have 
$$
\l_{S_{k}}(\beta^k_i)\le \l_{Y_{k}}(\alpha^k_i)\le \epsilon_k, 
$$
so it follows that
\begin{equation}\label{eq-lay}
\beta^k_i \subset \Thin_r(S_k),\quad\quad \text{when $\epsilon_k<2r$}.
\end{equation}
Moreover, if $r<r_M$ then $\Thin_r(S_k)[\beta^k]$ consists of disjoint annuli. This justifies the following definition.
\begin{definition}\label{definition-Z} Assuming $r<r_M$, we let $Z_k(r)$ denote the component of $S_k
\setminus \Thin_r(S_k)[\beta^k]$ which is homotopy equivalent to $\iota_k(Y_k)$.
\end{definition}

\begin{lemma}\label{lemma-elem-dok} The inclusion $Z_k(r)\subset \iota_k(Y_k)$ holds 
assuming $\epsilon_k<2r<2r_M$.
\end{lemma} 
\begin{proof} Let $\wh{Y}_k\subset Y_k$ be the subsurface whose boundary consists of the geodesics $\alpha^k_i$, $1\le i\le n$. Thus, $\wh{Y}_k$ is a surface with  geodesic boundary, and is homotopy equivalent to $Y_k$. It suffices to  show  $Z_k(r)\subset \iota_k(\wh{Y}_k)$.

Proof by contradiction. Assume $Z_k(r)$ is not contained in $\iota_k(\wh{Y}_k)$. Since $\iota_k(\wh{Y}_k)$ and $Z_k(r)$ are  homotopy equivalent essential subsurfaces, it follows that the boundary of $\iota_k(\wh{Y}_k)$ intersects $Z_k(r)$. But the boundary of $\iota_k(\wh{Y}_k)$ consists of the curves $\{\beta^k_i\}_{1\le i\le n}$,  so we conclude that some $\beta^k_i$ intersect $Z_k(r)$. From (\ref{eq-lay}) we know that $\beta^k_i \subset \Thin_r(S_k)$ which is a contradiction since $\beta^k_i$ intersect $Z_k(r)$.

\end{proof}

The following claim follows directly from the assumption that $\epsilon_k\to 0$ when $k\to \infty$.
\begin{claim}\label{claim-vise} Let $E\subset X$ be a compact set. There exits $k_0\in \N$ such that
$E\subset Y_k$ for every  $k>k_0$.
Moreover, we have 
\begin{equation}\label{eq-dok}
\lim\limits_{k\to \infty} d_{Y_{k}}(z,w)=d_{X}(z,w),
\end{equation}
for every $z,w\in X$. 
\end{claim}

\begin{lemma}\label{lemma-tvrd-dok} Let  $E\subset X_j$ be a compact subset.
There exist  $r_0<r_M$, and $k_0 \in \N$, such that for every $r<r_0$ the inclusion
$\iota_k(E)\subset Z_k(r)$ holds for  $k>k_0$ (the constants $r_0$ and $k_0$ depend on $E$).
\end{lemma}
\begin{proof} Since $E$ is compact it suffices to prove that for a fixed $z\in E$ the relation $\iota_k(z)\subset Z_k(r)$ holds for $k$ large enough.

Let $\gamma \subset X$ be a  homotopically non-trivial closed curve  which is  not homotopic to a power of a curve which bounds a puncture on $X$, and such that $z\in \gamma$. From Claim \ref{claim-vise} 
we find $k_0\in \N$ such  that $\gamma\subset Y_k$ for every $k>k_0$.
Set $D=1+\l_{X_{j}}(\gamma_j)$. Then by (\ref{eq-dok}) we see that that 
$\l_{Y_{k}}(\gamma)<D$ for $k$ large enough.

Let $C=C(D)$ be the constant from Lemma \ref{lemma-tvrd}, and set $r_0=\frac{r_M}{C}$.
If $\iota_k(z) \notin Z_k(r)$ for some $r<r_0$, then $\iota_k(z) \in \Thin_r(S)[\beta^k]$. Thus, the curve $\iota_k(\gamma)$ intersects  $\Thin_r(S)[\beta^k]$, and from Lemma \ref{lemma-tvrd} we find that  $r\ge \frac{r_M}{C}=r_0$. This proves that $\iota_k(z) \in Z_k(r)$ for every $r<r_0$, and every $k>k_0$.
\end{proof}

\subsection{Proof of Lemma \ref{lemma-fin}} Let 
\begin{equation}\label{eq-fin-3}
K_k=\exp\big(d_{\T}(S^1_k,S^2_k)\big),
\end{equation}
and set
\begin{equation}\label{eq-fin-3-11}
K=\limsup_{k\to \infty} K_k.
\end{equation}
Then there exists a $K_k$-quasiconformal homeomorphism $f_k:S^1_k\to S^2_k$. 
Denote by  $Y^j_k \subset X_j$ the  subsurfaces such that $J_{\epsilon_{k}}(X_j)=Y^j_k$, and  by $\iota^j_k:Y^j_k\to S^j_k$ the corresponding conformal embeddings, for $j=1,2$. 
The surfaces $Y^1_k$ and $Y^2_k$ may  not be the domain and the range of the map $(\iota^2_k)^{-1}\circ f_k\circ \iota^1_k$. However, we can slightly trim these subsurfaces so that this map is well defined. 

Suppose $\epsilon_k<2r<2r_M$. By $Z^j_k(r)\subset S^j_k$ we denote the corresponding subsurface from Definition \ref{definition-Z}. Define
$$
G^1_k(r)=(\iota^1_k)^{-1}(Z^1_k(r)).
$$

\begin{claim}\label{claim-posl} For every $0<r<r_M$, there exists $k_0\in \N$ such that the map $g_k=(\iota^2_k)^{-1}\circ f_k\circ \iota^1_k$ is well defined  on $G^1_k(r)$ for $k>k_0$.
\end{claim}
\begin{proof} Let $k_0$ be such that  the map $f_k$ is  $(K+1)$-quasiconformal when $k>k_0$.
Let $C=C(K+1)$ be the constant from  Lemma \ref{lemma-mozda-5}. By applying (\ref{eq-mozda-2}) to the inverse map $f^{-1}_k$  we conclude that 
\begin{equation}\label{eq-mozdana}
f_k\big(Z^{1}_{k}(r)  \big) \subset Z^{2}_{k}(C^{-1}r),
\end{equation}
when $k>k_0$. On the other hand, from Lemma \ref{lemma-elem-dok} we conclude that 
$$
Z^2_k(C^{-1}r)\subset \iota^2_k(Y^2_k),
$$
when $\epsilon_k<2C^{-1}r<2r_M$. Thus, for every $z\in G^1_k(r)$, the point $(f_k\circ \iota^1_k)(z)$ lies in the image  $\iota^2_k(Y^2_k)$, and we are done.
\end{proof}

Suppose $E\subset X_1$ is a compactly contained domain. Let $r_0$ be the constant from
Lemma \ref{lemma-tvrd-dok}. Then, by the same lemma we find that for every $0<r<r_0$  the inclusion  $E\subset G^1_k(r)$ holds  for   $k$ is large enough. We  conclude that the map $g_k:E\to g_k(E)$ is a well defined $K_k$-quasiconformal homeomorphism for every large enough $k$.. 

Since this is true for every $E\subset X_1$, it follows that  $g_k$ converges to a $K$-quasiconformal homeomorphism $g:X_1\to  X_2$ which respects the markings.  Thus, 
$$
d_{\T}(X_1,X_2)\le \log K= \log \big( \limsup_{k\to \infty} K_k\big).
$$
Combining this with (\ref{eq-fin-3}) proves Lemma \ref{lemma-fin}.

\section{Teichm\"uller discs }  In  \cite{Markovic} we  developed the general framework for analysing when  $d_C=d_\T$ on Teichm\'uller discs on Riemann surfaces of finite analytic type which are generated by Jenkins-Strebel holomorphic quadratic differentials. The purpose of the following three sections is to see that the exact  framework works   for Riemann surfaces of finite topological type as well.

Let $S$ denote a topologically finite Riemann surface (every such surface arises  by removing a finite number of points and discs from a closed surface). By $\QMod(S)$ we denote the group of  quasiconformal self-homeomorphisms of $S$ up to homotopy (relative to the boundary at infinity).
Then $\QMod(S)$ naturally acts on $\T(S)$ as the group of biholomorphic  automorphisms.

\subsection{Teichm\"uller discs} 

A  holomorphic  quadratic differential $\varphi \in \QD(S)$ induces a holomorphic map from the upper half plane $\Ha$  to $\Belt_1(S)$ given by 
$\lambda \to \left( \frac{\imu-\lambda}{\imu+\lambda} \right) \frac{| \varphi |}{\varphi}$. In turn this defines the holomorphic embedding  $\tau^{\varphi}:\Ha \to \T(S)$ by letting 
\begin{equation}\label{ponovimo}
\tau^{\varphi}(\lambda)=\left[  \left( \frac{\imu-\lambda}{\imu+\lambda} \right) \frac{| \varphi |}{\varphi}   \right] \in \Teich(S).
\end{equation}
Note that $\tau^{\varphi}(\imu)=S$. We refer to the map $\tau^{\varphi}$ as the Teichm\"uller disc.

Let $S(\lambda)=\tau^{\varphi}(\lambda)$. The induced quasiconformal map $S \to S(\lambda)$, with the dilatation 
$ \left( \frac{\imu-\lambda}{\imu+\lambda} \right) \frac{| \varphi |}{\varphi}$  is affine in local coordinates corresponding to $\varphi$ and the corresponding terminal quadratic differential $\varphi(\lambda) \in \QD\big(S(\lambda)\big)$. If $z$ is the local parameter on $S$ such that $\varphi=dz^2$ then the quasiconformal map is of the form 
$z \, \to  \, x+\lambda y$, where $z=x+y\imu$.

\subsection{Stabilisers and Jenkins-Strebel differentials}  We define the  subgroup $\Stab(\tau^{\varphi})<\Aut(\Ha)$, which stabilises  the Teichm\"uller disc $\tau^{\varphi}$, as follows.  Let $A\in Aut(\Ha)$. Then  $A\in \Stab(\tau^{\varphi})$ if the Riemann surfaces underlying $\tau^{\varphi} \circ A$ and  $\tau^{\varphi}$ agree.  The induced monomorphism  $\tau^{\varphi}_{*}: \Stab(\tau) \to \QMod(S)$ detects the difference between the markings of $\tau^{\varphi} \circ A$, and $\tau^{\varphi}$, respectively.

\begin{definition} 
A differential $\varphi \in \QD(S)$ is called a Jenkins-Strebel differential if $\varphi$ induces a decomposition of  $S$ into a finite number of annuli $\Pi_j$, $j=1,...,k$,   foliated by closed horizontal trajectories of $\varphi$.  
\end{definition}

\begin{remark} Suppose that $\gamma\subset S$ is a peripheral geodesic. If $\varphi \in \QD(S)$ is a Jenkins-Strebel differential then one of the annuli $\Pi_j$ is homotopic to $\gamma$.
\end{remark}

Let $\gamma_1,...\gamma_k$ be a collection of disjoint   simple closed curves on $S$ homotopic to $\Pi_j$'s. By $m_j$ we denote the conformal modulus of $\Pi_j$. If  $m_j$'s have rational ratios  we call $\varphi$ a rational Jenkins-Strebel differential.  By $T_{\gamma_{j}} \in \QMod(S)$ we denote  the Dehn twist about $\gamma_j$.  The following lemma  is well known  (see Lemma 9.7 in  \cite{McMullen} or Proposition 2.1 in \cite{Markovic}).
\begin{lemma}\label{lemma-varphi} Let $\varphi$ be  a  rational Jenkins-Strebel differential and let
\begin{equation}\label{t}
t=\lcm\{m^{-1}_1,...,m^{-1}_k \}, 
\end{equation}
where $\lcm$ stands for the lowest common multiple.  Set  $A_t(\lambda)=\lambda+t$, $\lambda \in \Ha$.  Then $A_t \in \Stab(\tau^{\varphi})$, and 
$\tau^{\varphi}_{*}(A_t)=T \in \QMod(S)$ is the product of the (commuting) Dehn twists $T=T^{n_{1}}_{\gamma_{1}} \cdot \cdot \cdot T^{n_{k}}_{\gamma_{k}} $, where $n_{j}=m_j t$.
\end{lemma}

\section{Carath\'eodory metric on Teichm\"uller discs}\label{sec-eqiv} Let $S$ continue to denote a topologically finite Riemann surface. Let $\varphi \in \QD(S)$, and consider the Teichm\"uller disc $\tau^{\varphi}$. By Royden's theorem $\tau^{\varphi}$  is an isometric embedding of $(\Ha,d_\Ha)$  into $(\Teich(S),d_\T)$. Thus, from Lemma \ref{lemma: basic lemma 1} we conclude that $d_C=d_\T$ on $\T(S)$  if an only if there exists a holomorphic map  $\Hol:\Teich(S) \to \Ha$ such that  $\big(\Hol \circ  \tau^{\varphi}\big)  \in \Aut(\Ha)$.

\begin{lemma}\label{lemma-Hol} Fix a Jenkins-Strebel differential $\varphi \in \QD(S)$ and let $\Hol:\Teich(S) \to \Ha$ be any holomorphic map such that  $\big(\Hol\circ \tau^{\varphi}\big)(\lambda)=\lambda$, for every $\lambda \in \Ha$. The derivative of $\Hol$ at $S(\lambda)=\tau^{\varphi}(\lambda)$  is given by
\begin{equation}\label{hol-derivative}
d\Hol(\nu)=  \frac{-2\imu\Img(\lambda)}{||\varphi(\lambda)||_{1}}\, \int\limits_{S(\lambda)} \nu \varphi(\lambda),
\end{equation}
where $\nu \in \Belt\big(S(\lambda)\big)$ represents a tangent vector to $\Teich(S)$ at the point $S(\lambda)$, and $||\varphi (\lambda)||_1$ denotes  the $L^{1}$-norm of  $\varphi(\lambda) \in \QD\big(S(\lambda)\big)$. 
\end{lemma}
\begin{remark} When $S$ is analytically finite we do not need to assume that $\varphi$ is a Jenkins-Strebel differential (see  Proposition 3.1 in \cite{Markovic}).
\end{remark}

\begin{proof}  Fix $\lambda\in \Ha$, and let $A:\Ha \to \D$ be the M\"obius map given by
$A(z)=\frac{z-\lambda}{z-\overline{\lambda}}$.
Note $A(\lambda)=0$. Set $\Fol=A \circ \Hol$. Then $\Fol:\Teich(S) \to \D$ is holomorphic.
Moreover, since $\Fol\big(\Teich(S)\big) \subset \D$, and $\Fol\big(S(\lambda)\big)=0$, it follows from the Schwarz lemma and the Royden's theorem that 
\begin{equation}\label{eq-kont}
\left| d\Fol([\nu]) \right| \le ||\nu||_{\infty}
\end{equation}
for every $\nu \in \Belt\big(S(\lambda)\big)$. Therefore, $d\Fol:T_S\T(S)\to \C$ is a linear functional of norm at most one.

On the other hand, let $f:\D\to \Teich_{g,n}$ be given by $f=\tau^{\varphi} \circ A^{-1}$. Then
$f$ is the Teichm\"uller disc given by 
$$
f(\eta)=\left[  \eta \frac{| \varphi(\lambda) |}{\varphi(\lambda)}   \right] \in \Teich(S),
$$
and $\big( \Fol \circ f \big) (\eta)=\eta$ for every $\eta\in \D$. Thus
$d\Fol\big(\frac{df}{d \eta} \big) =1$.
Since $\frac{df}{d\eta}=\frac{|\varphi(\lambda)| }{\varphi(\lambda)}$, we get
\begin{equation}\label{hol-derivative-proof-3}
d\Fol\left( \left[ \frac{|\varphi(\lambda)| }{\varphi(\lambda)} \right] \right) =1.
\end{equation}

Replacing (\ref{eq-kont}) and (\ref{hol-derivative-proof-3}) in Proposition \ref{prop-tame} we get
\begin{equation}\label{hol-derivative-proof-4}
d\Fol(\nu)=   \frac{1}{||\varphi(\lambda)||_{1} } \int\limits_{S(\lambda)} \nu  \varphi(\lambda),
\end{equation}
for every $\nu \in \Belt\big(S(\lambda)\big)$ (we note that Proposition \ref{prop-tame} is the only place in this proof where we use that $\varphi$ is Jenkins-Strebel). From $\Fol=A \circ \Hol$ we get 
$$
d \Hol(\nu)=\frac{1}{A'(\lambda)} \, d\Fol (\nu).
$$
Replacing  $A'(\lambda)=\frac{1}{-2\imu\Img(\lambda)}$
in the previous identity yields the proof.
\end{proof}

\subsection{Equivariant holomorphic functions}   Assume  $\varphi \in \QD(S)$  to be  a  rational  Jenkins-Strebel differential.  In Lemma \ref{lemma-varphi} we constructed the  Dehn twist  $T \in \QMod(S)$ which lies in the image of the homomorphism $\tau^{\varphi}_{*}: \Stab(\tau) \to \QMod(S)$. In the proof of the next lemma
we average a holomorphic function  $\Fol:\Teich(S) \to \Ha$ (in a suitable sense) over the cyclic group generated by $T$ to obtain an equivariant holomorphic function $\Hol:\Teich(S) \to \Ha$.

\begin{lemma}\label{lemma-equiv-10} Suppose  $\varphi \in \QD(S)$ is a rational  Jenkins-Strebel differential.  
If $d_C=d_\T$ on the Teichm\"uller disc $\tau^{\varphi}(\Ha)\subset \T(S)$ then  there exists a holomorphic function  $\Hol:\Teich(S) \to \Ha$ with the following properties
\begin{enumerate}
\item   \,\,\, $\big(\Hol \circ \tau^{\varphi}\big)(\lambda)=\lambda$,  for every  $\lambda \in \Ha$,
\vskip .1cm
\item \,\,\,  $\big( \Hol \circ T \big)(S')=\Hol(S')+t$,  for every  $S' \in \Teich(S)$, where 
$T=\tau^{\varphi}_{*}(A_t)$ is the Dehn twist from  Lemma \ref{lemma-varphi}. 
\end{enumerate}
\end{lemma}
\begin{remark}
When $S$ is of finite analytic type this was proved in Lemma 3.1 in \cite{Markovic}.
\end{remark}
\begin{proof} 
Recall  from Lemma \ref{lemma-varphi} the automorphism $A_t \in \Stab(\tau^{\varphi})$ given by $A_t(\lambda)=\lambda+t$, $\lambda \in \Ha$ (the number $t$ is given by (\ref{t})).  The Dehn multi-twist $T \in \QMod(S)$ is given by 
$\tau^{\varphi}_{*}(A_t)=T$, and we have $T\circ \tau^{\varphi}=\tau^{\varphi}\circ A_t$.

Since $d_C=d_\T$ on  $\tau^{\varphi}(\Ha)\subset \T(S)$, from Lemma \ref{lemma: basic lemma 1} we conclude that  there exists a holomorphic map  $\Psi:\Teich(S) \to \Ha$ such that  $\big(\Psi \circ  \tau^{\varphi}\big)=\id$ on $\Ha$. Define $\Psi_n:\T(S)\to \Ha$ by $\Psi_n=\Psi\circ T^n-nt$.
Since $T\circ \tau^{\varphi}=\tau^{\varphi}\circ A_t $ it follows that for every $n$ the equality
\begin{equation}\label{eq-n}
\Psi_n\circ \tau^{\varphi}=\id,
\end{equation}
holds on $\Ha$.
Furthermore, $\Psi_n$ satisfies the recurrence formula
\begin{equation}\label{eq8}
\Psi_n\circ T=\Psi_{n+1}+t.
\end{equation}

From the Schwartz lemma and (\ref{eq-n}) we conclude $d_\Ha(\Psi_n(S'),i)\leq d_\T(S',S)$ for every $S'\in \T(S)$, and every $n\in \N$. This implies that there exist  constants $C,q>0$, depending on $S'$, but independent of $n$, such that
\begin{equation}\label{eq-n1}
-C< \Real\big( \Psi_n(S')\big)<C, \quad \quad \Img(\Psi_n(S'))>q.
\end{equation}

Let $L:l^\infty(\N)\rightarrow\C$   be a Banach limit. Recall that $L$ satisfies the following properties:
\begin{enumerate}
\item $L(\{a_n\}_{n\in\N})=a$, if $a_n\rightarrow a$, 
\vskip .1cm
\item $L((a_n)_{n\in\N})=L(\{a_{n+1}\})_{n\in\N})$.
\end{enumerate}
Set
$$
\Hol(S')=L\left(  \{\Psi_n(S')\}_{n\in\N} \right).
$$
From (\ref{eq-n1}) we conclude that the range of  $\Hol$ is $\Ha$. Moreover,  the map $\Hol:\T(S)\to \C$ is holomorphic as a composition of two holomorphic maps. Now, from (\ref{eq-n}), and the property (1) of $L$, we conclude that  $\Hol \circ \tau^{\varphi}=\id$ on $\Ha$. Furthermore, combining (\ref{eq8}) with  the property (2) of $L$ implies that 
$\Hol\circ T=\Hol+t$. This completes the proof.
\end{proof}

\section{Mapping the polyplane $\Ha^{k}$ to $\Teich(S)$ } \label{sec-poly}

\subsection{Teichm\"uller polyplanes}

Select a rational Jenkins-Strebel quadratic differential  $\varphi \in \QD(S)$  and   let $h_1,...,h_{k}>0$ denote the heights  (with respect to the $|\varphi|$ singular metric) of  the corresponding annuli $\Pi_j$.  We let $\Ha^k=\Ha \times \cdot \cdot \cdot \times \Ha$ denote the $k$-fold product of the upper half plane $\Ha$, and 
let $\lambda=(\lambda_1,...,\lambda_k)$ denote the coordinates on $\Ha^k$.  Define $\Embo:\Ha^{k} \to \Belt_1(S)$ by letting
$$
\Embo(\lambda)=\left( \frac{\imu-\lambda_j}{\imu+\lambda_j} \right) \frac{| \varphi |}{\varphi},    
$$
on each $\Pi_j$.  This yields the map $\Emb:\Ha^{k} \to \Teich(S)$  by  letting $\Emb (\lambda)=\left[ \Embo(\lambda)\right]$. We say that  $\Emb$ is the polyplane mapping corresponding to  $\varphi$. 
The map $\Embo$ is clearly holomorphic and thus $\Emb$ is holomorphic as well.  Observe  that the restriction of $\Emb$ on the diagonal in 
$\Ha^{k}$ is the Teichm\"uller disc $\tau^{\varphi}$.

The  new marked Riemann surface $\Emb (\lambda)=S(\lambda)$ comes equipped with the quadratic differential $\varphi(\lambda)$ which is
the unique  Jenkins-Strebel differential in $\QD\big(S(\lambda)\big)$ which induces a decomposition of $S(\lambda)$ 
into annuli $\Pi_j(\lambda)$   (swept out by closed horizontal trajectories homotopic to $\gamma_j$)  such that the height $h_j(\lambda)$ of 
$\Pi_j(\lambda)$ is given by $h_j(\lambda)=\Img(\lambda_j) h_j$. 
  
Let $m_j$ denote the conformal modulus of the annulus $\Pi_j$. Observe that $\Emb$ conjugates the translation  $\lambda \to \big( \lambda+(0,..,m^{-1}_j,...,0) \big)$ to the twist $T_{\gamma_{j}} \in \QMod(S)$, that is 
$$
\Emb\big( \lambda+(0,..,m^{-1}_j,...,0) \big)=\left( T_{\gamma_{j}} \circ \Emb\right)(\lambda),
$$
for every $\lambda \in \Ha^{k}$. This yields the equality

\begin{equation}\label{conj-twist}
\Emb\big( \lambda+(t,...,t) \big)=\left( T \circ \Emb\right)(\lambda),
\end{equation} 
where $t$ is given by (\ref{t}) and $T\in \QMod(S)$ is the corresponding Dehn twist
from Lemma \ref{lemma-varphi} above.

\subsection{Criterion for $d_C=d_\T$}
Suppose $\Hol:\Teich(S) \to \Ha$ is a holomorphic function satisfying the conditions $(1)$ and $(2)$ from  Lemma \ref{lemma-equiv-10}.  Let $f= \Hol \circ \Emb$. Then $f:\Ha^{k} \to \Ha$. Moreover, the values of $f$ on the diagonal in $ \Ha^{k}$ are given by

\begin{equation}\label{f-1}
f(\eta,\eta,...,\eta)=\eta, 
\end{equation}
for every  $\eta \in \Ha$.

Let $\gamma_j$  stand for  simple closed curves on $S$ homotopic to the  corresponding annuli $\Pi_j$ (which are swept out by closed horizontal trajectories of $\varphi$). Taking into the account  (\ref{f-1}) and (\ref{hol-derivative}), an applying Proposition 5.1in \cite{Markovic},  we get that
for $\lambda=(\eta,...,\eta)$  the following holds
$$
\frac{\partial{f}}{\partial{\lambda_{j}}} (\lambda)=\alpha_j,
$$
where 
$$
\alpha_j=  \frac{1}{||\varphi(\lambda)||_{1}}  \, \int\limits_{\Pi_{j}(\lambda)} |\varphi (\lambda)| =\frac{1}{||\varphi||_{1}}  \, \int\limits_{\Pi_{j}} |\varphi |.
$$
The following is Theorem 5.1 in \cite{Markovic}. 
\begin{theorem}\label{thm-Hol}  Let $\varphi$ be a rational  Jenkins-Strebel differential that induces  $S$ to  decompose into exactly two (non-degenerate) annuli swept out by closed horizontal trajectories of $\varphi$.  Then  any holomorphic function   $\Hol :\Teich(S) \to \Ha$   satisfying the conditions $(1)$ and $(2)$ from Lemma \ref{lemma-equiv-10} has the property that its   restriction  on the set $\Emb(\Ha^2) $ is  given by the formula  
$$
\big(\Hol \circ \Emb\big)(\lambda_1,\lambda_2)=\alpha_1 \lambda_1 +\alpha_2 \lambda_2,
$$
for every  $(\lambda_1,\lambda_2) \in \Ha^2$.
\end{theorem}

\section{$L$-shaped 3-punctured disc} In Section 7 in \cite{Markovic} we constructed an $L$-shaped 5-punctured sphere by doubling  an $L$-shaped Euclidean polygon. In this section we repeat the same construction except that here we glue the two copies of the polygon along five of its edges (to construct an $L$-shaped 5-punctured sphere one needs to glues the two copies along all six edges).

\subsection{$L$-shaped polygons}  
We make the standing assumptions that 
$$
a>0, \quad b\ge 0,\quad \text{and $0<q<1$}.
$$ 
Let $L=L(a,b,q)$  denote the $L$-shaped polygon as in Figure \ref{fig:Ls1}.  By $P_k$, $k=1,...,5$, we denote the vertices at which the interior angle is $\pi/2$ and by $Q$ the remaining vertex at which the interior angle is $3\pi/2$. 

Let $S=S(a,b,q)$ be the union of two copies of $L(a,b,q)$ with all edges identified except the bottom edge. Formally, $S$ is defined as a half-translation surface. As a Riemann surface $S$ is a 3-punctured disc. The points $P_1,P_2,P_3$ are the three punctures. The boundary of the disc is the curve  parametrised by the union of two copies of the interval connecting the points $P_5$ and $P_1$  (this is the bottom edge of $L(a,b,q)$). Thus, the points $P_5$ and $P_1$ live on the boundary of the disc. We call $S$  a L-shaped 3-punctured disc.

The $(2,0)$ form $dz^2$ lives on the polygon $L$. After doubling, the two copies of the form $dz^2$ glue together along the edges of the L-shaped 3-punctured disc $S$  to form the holomorphic quadratic differential $\psi=\psi(a,b,q)$ on $S$. The differential $\psi$ has the first order poles at the points $P_k$ and the first order zero at  $Q\in S$.  
Moreover, if $b>0$ then $\psi$ is a Jenkins-Strebel differential and $S$ decomposes into two non-degenerate annuli $\Pi_1$ and $\Pi_2$ swept out by closed horizontal trajectories  of $\psi$. (If $b=0$ then $\psi$ has first order poles at the points $P_1,P_2,P_4,P_5$ and no zeroes).

\subsection{$L$-shaped 3-punctured discs}

Fix once for all  rational numbers $a_0,b_0,q_0\in \Q$, where $a_0,b_0>0$ and $q_0\in(0,1)$. Set $S_0=S(a_0,b_0,q_0)$ and  $\psi_0=\psi(a_0,b_0,q_0)$.  Note that $\psi_0$ is a rational Jenkins-Strebel differential.

Each Riemann surface $S(a,b,q)$ has the standard marking $S_0\rightarrow S$ which restricts to the identity map on the boundary which is the union of two copies of the interval $\overline{P_5P_1}$ with endpoints identified. 
The corresponding family of standardly marked  surfaces $S(a,b,q)$ varies   continuously in $\Teich(S_0)$ when  $a,b,q$ vary continuously, and  is a real 3-dimensional locus   in $\Teich(S_0)$.

\begin{figure}[t]
\begin{center}
\includegraphics[height=5cm]{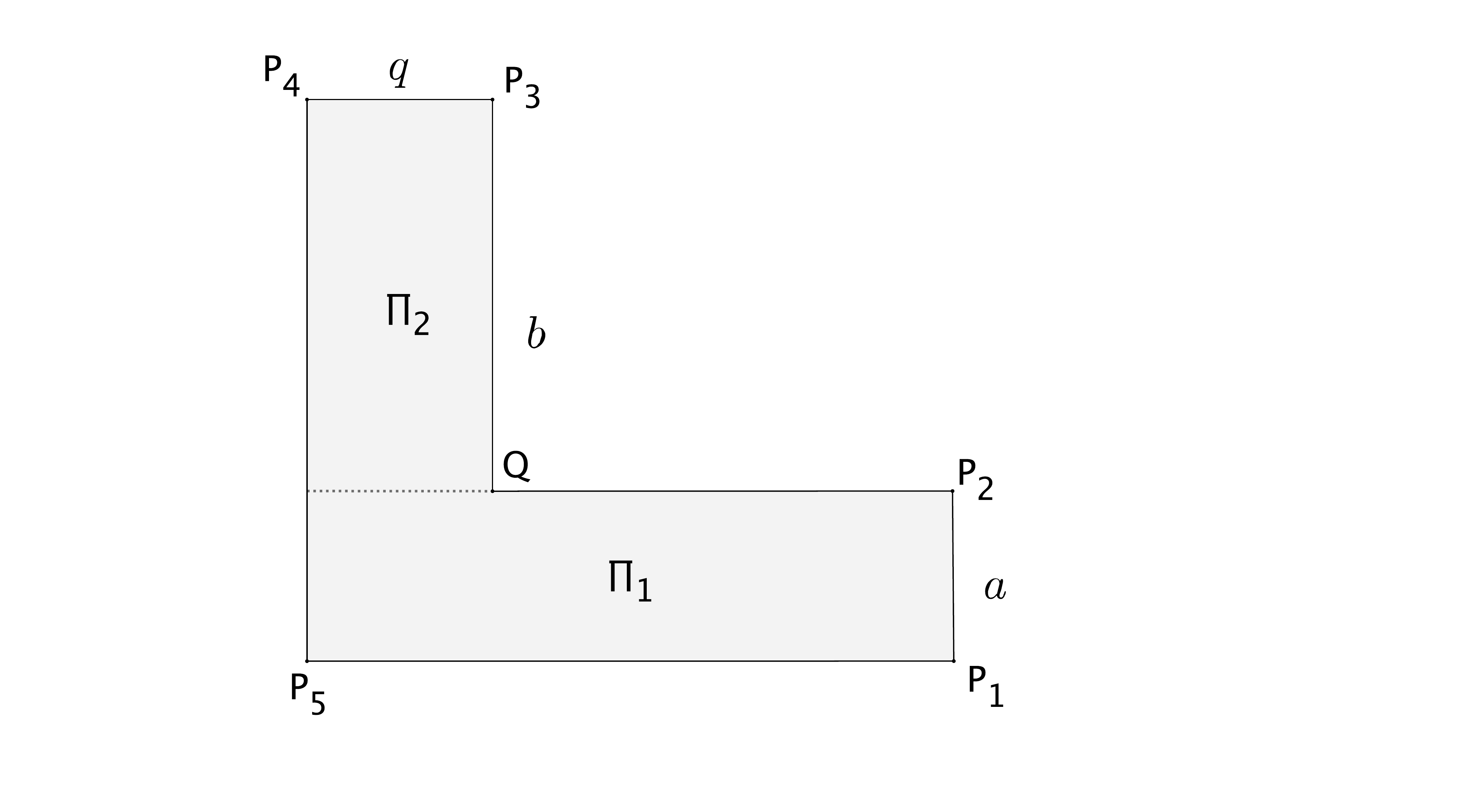}
\caption{$L$-shaped polygon $L(a,b,q)$ }
\label{fig:Ls1}
\end{center}
\end{figure}

\begin{remark} We simplify the notation we also use $S(a,b,q)$ to denote the corresponding 3-punctured disc with the standard marking. 
\end{remark}

\begin{remark} If as in \cite{Markovic} we glue two copies of $L(a,b,q)$ along all six sides then $\big(S(a,b,q), \psi(a,b,q)\big)$  is a real 3-dimensional locus of the cotangent bundle over 
$\Teich(S_0)$ while  the  family of standardly marked Riemann surfaces  $S(a,b,q)$  is a real 2-dimensional locus in $\Teich(S_0)$. 
\end{remark}

\subsection{The Carath\'eodory metric on $\tau^{\psi_{0}}$ } The following lemma shows that if $d_C=d_\T$ on $\tau^{\psi(a_0,b_0,q_0)}$ then there exists  a holomorphic function   $\Fol :\Teich(S_0) \to \Ha$ which is linear in  $a$ and $b$ when restricted to the locus $S(a,b,q_0)$.

\begin{lemma}\label{lemma-texnika}  Fix $q_0 \in (0,1)$ and suppose $d_C=d_\T$ on $\tau^{\psi(a_0,b_0,q_0)}$. Then there exists  a holomorphic function   $\Fol :\Teich(S_0) \to \Ha$ such that 
$$
\Fol\big(S(a,b,q_0)\big)=\left( a+bq_0   \right)\imu,
$$
for every $a>0$ and $b\ge 0$. 
\end{lemma}

\begin{proof} 

We 
let $\Emb_0:\Ha^2 \to \Teich(S_0)$ be the polyplane mapping corresponding to $\psi_0$. Set
$$
\alpha_1=  \frac{1}{||\psi_0||_{1}} \, \int\limits_{\Pi_{1}} |\psi_0 |=\frac{b_0q_0}{a_0+b_0q_0},
$$
and 
$$
\alpha_2= \frac{1}{||\psi_0||_{1}} \,     \int\limits_{\Pi_{2}} |\psi_0 |=\frac{a_0}{a_0+b_0q_0}.
$$
Since $a_0,b_0,q_0 \in \Q$ it follows that $\psi_0$ is a Jenkins-Strebel differential. So, if $d_C=d_\T$ on $\tau^{\psi_{0}}$  then  by Theorem \ref{thm-Hol} there exists a  holomorphic function  $\Hol:\Teich(S_0) \to \Ha$ such that
$$
\Hol\big(\Emb_0(\lambda_1, \lambda_2) \big)=\alpha_1 \lambda_1+\alpha_2 \lambda_2,
$$
where $(\lambda_1,\lambda_2) \in \Ha^2$.  We find that for any $a$ and $b$ the equality 
$$
\Emb_0\left(\frac{b}{b_0} \imu, \frac{a}{a_0} \imu\right)=S(a,b,q_0),
$$
holds. The last two formulas yield the following equality
\begin{equation}\label{izraz}
\Hol\big(S(a,b,q_0)\big)=\left( \frac{a+bq_0}{a_0+b_0q_0}   \right)\imu.
\end{equation}
Set 
$$
\Fol=(a_0+q_0b_0)\Hol.
$$
Then $\Fol:\Teich(S_0)$ is holomorphic and from (\ref{izraz}) we see that $\Fol$   satisfies the stated equality. To derive this formula we assumed $b>0$. Since $S(a,b,q) \to S(a,0,q)$  in $\Teich(S_0)$ when $b \to 0$,  the formula also holds for $b=0$. 
\end{proof}

\subsection{The derivative of $\Fol$ on the locus $S(a,b,q_0)$.}
We are now able to compute the derivative of the map $\Fol$ at every point of the locus $S(a,b,q_0)$.
    
\begin{lemma}\label{lemma-ab}    
For all $a>0$, $b\ge 0$,  we have
\begin{equation}\label{eq-ab}
d\Fol(\nu)= -\imu \int\limits_{S(a,b,q_0)} \nu\psi(a,b,q_0)
\end{equation}
at the point $S(a,b,q_0)$, for all $\nu\in\BD(S(a,b,q_0))$.
\end{lemma}
\begin{proof} Note that when we restrict $\Emb_0$ to the slice 
$\left(\frac{b}{b_0}\lambda,\frac{a}{a_0}\lambda \right)\subset \Ha^2$ we get the Teichm\"uller disc $\tau^{\psi(a,b,q_0)}$. That is, we have the equality
$$
\tau^{\psi(a,b,q_{0})}(\lambda)=\Emb_0\left(\frac{b}{b_0}\lambda,\frac{a}{a_0}\lambda \right),
$$
for every $\lambda\in \Ha$. Combining this with Lemma \ref{lemma-texnika} we get 
$$
(\Fol\circ \tau^{\psi(a,b,q_{0})})(\lambda)=(\Fol\circ \Emb_0)\left(\frac{b}{b_0}\lambda,\frac{a}{a_0}\lambda \right)=(a+bq_0)\lambda.
$$
Applying  Lemma \ref{lemma-Hol} yields the equality
$$
d\Fol(\nu)= \frac{-2 \imu (a+bq_0) }{||\psi(a,b,q_0)||_{1} }  \int\limits_{S(a,b,q_0)} \nu\psi(a,b,q_0).
$$
Since $||\psi(a,b,q_0)||_{1}=2(a+bq_0)$ the previous equality proves the lemma.
\end{proof}

\section{Two paths in $\T(S_0)$}

Our  endgame  is to show that $\Fol$ is not smooth at a  point in the closure of  $\Emb_0(\Ha^2)\subset \T(S_0)$ which contradicts  the fact that $\Fol$ is holomorphic. As indicated in the introduction, we are going to investigate two paths $\sigma_1(t),\sigma_2(t)$ in $\T(S_0)$. The path $\sigma_1(t)$ is smooth while $\sigma_2(t)$ lies in the polyplane $\Emb_0(\Ha^2)$. 
We rely on the results from Section 9 of \cite{Markovic}.

\subsection{The first path} We have:

\begin{definition} We define $\sigma_1:[0,q_0)\to \T(S_0)$ by letting $\sigma_1(t)=S(a_0,0,q_0-t)$.
\end{definition}

Recall that the standard marking $S_0\rightarrow S(a_0,0,q_0-t)$ restricts to the identity on the boundary of the disc. Here we parametrise the boundary of both surfaces by the union of two copies of the interval $\overline{P_5P_1}$ with endpoints identified. 

\begin{lemma}\label{lemma-kraj-1} The path $\sigma_1(t)$ is  infinitely differentiable.
\end{lemma}
\begin{proof} The marked surfaces $S(a_0,0,q_0-t)$ arise from the marked surface  $S(a_0,0,q_0)$ by smoothly varying one of its punctures while keeping the other two punctures, and the boundary of the disc, fixed. This implies that $\sigma_1$ is smooth.
\end{proof}

\begin{figure}[t]
\begin{center}
\includegraphics[height=5cm]{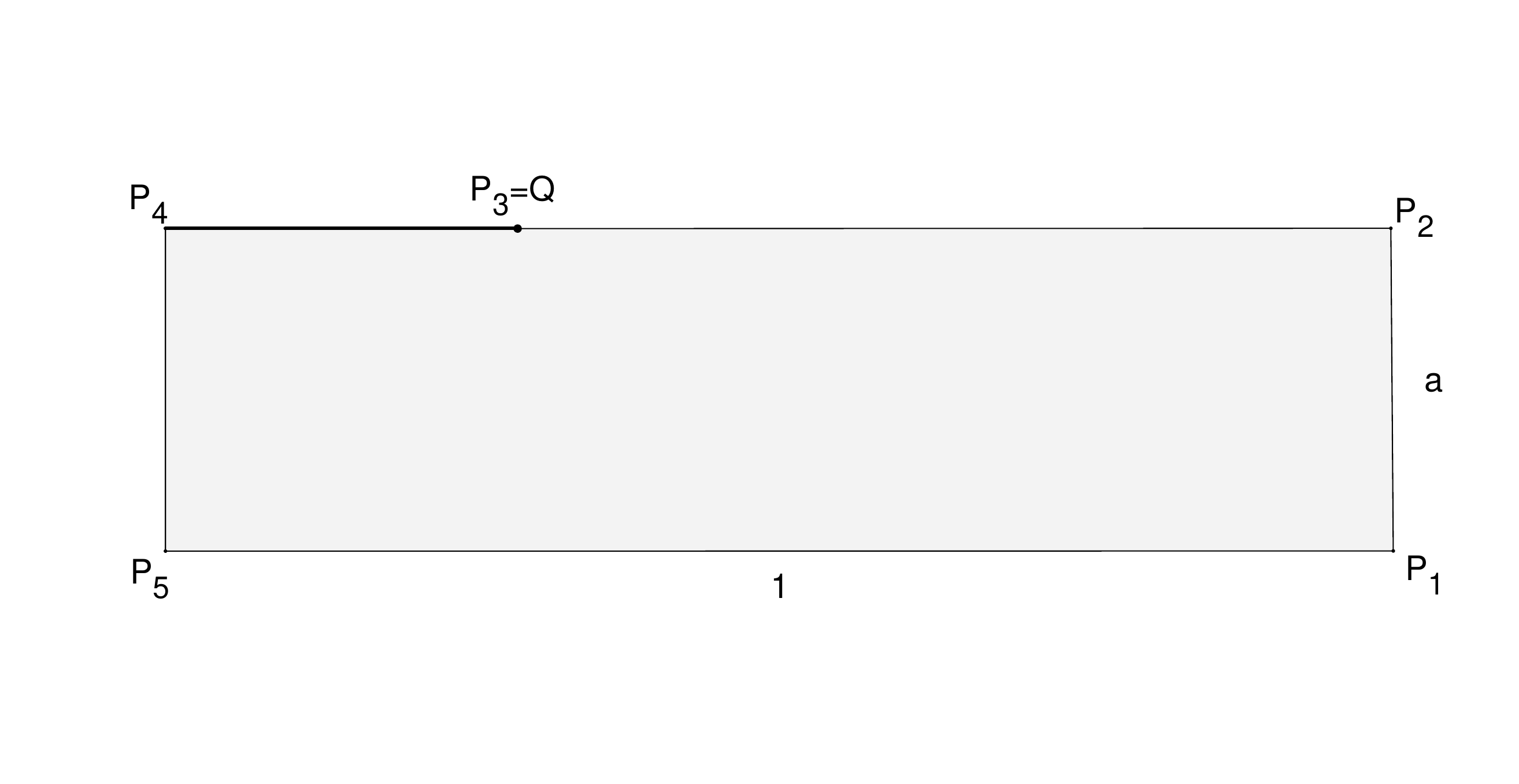}
\caption{$L(a_0,0,q_0-t)$}
\label{fig:Lr1}
\end{center}
\end{figure}

\subsection{ Schwarz-Christoffel maps}\label{subsection-SC} 
Let $\zeta,\lambda$ and $r\ge 0$ be such  that $-1<\zeta-r \le \zeta<\lambda<1$.  
Set 
$$
F(z)=J^{-1}\int\limits_{1}^{z} \, \frac{ \sqrt{w-\zeta+r} }{\sqrt{w+1} \sqrt{w-\zeta} \sqrt{w-\lambda} \sqrt{w-1}} \, dw  .
$$
Then  $F(\lambda,\zeta,\zeta-r)=F$ conformally maps  $\Ha$ onto  the polygon $L(a,b,q)$ where  
$a=\frac{A}{J}$,  $b=\frac{B}{J}$, $q=\frac{Q}{J}$, and
\begin{align*}
&A=A(\lambda,\zeta,r)=-\imu \, \int\limits_{-\infty}^{-1} \,  \frac{ \sqrt{x-\zeta+r} }{    \sqrt{x+1} \sqrt{x-\zeta} \sqrt{x-\lambda} \sqrt{x-1}}  \, dx  \\
&B=B(\lambda,\zeta,r)=-\imu \, \int\limits_{\zeta-r}^{\zeta} \, \frac{ \sqrt{x-\zeta+r} }{    \sqrt{x+1} \sqrt{x-\zeta} \sqrt{x-\lambda} \sqrt{x-1}}  \, dx   \\ 
&J= J(\lambda,\zeta,r)=\int\limits_{1}^{\infty} \, \frac{ \sqrt{x-\zeta+r} }{    \sqrt{x+1} \sqrt{x-\zeta} \sqrt{x-\lambda} \sqrt{x-1}} \, dx,   \\
&Q=Q(\lambda,\zeta,r)=- \, \int\limits_{\zeta}^{\lambda} \,\frac{ \sqrt{x-\zeta+r} }{    \sqrt{x+1} \sqrt{x-\zeta} \sqrt{x-\lambda} \sqrt{x-1}}  \, dx. 
\end{align*}
Moreover,  $F(\infty)=P_1$, $F(-1)=P_2$, $F(\zeta)=P_3$, $F(\lambda)=P_4$, and  $F(\zeta-r)=Q$. 
(Note that $b=0$ if and only if $r=0$.)

The following lemma collates some of the results proved in the first three subsections of Section 9 in \cite{Markovic}. In particular, the formula (\ref{r-t}) below is the formula (32) in \cite{Markovic}.

\begin{lemma}\label{lemma-novosti-0} There exist continuous functions $\lambda,\zeta,r,a,b:[0,q_0)\to \R$ with the following properties for every $t\in[0,q_0)$
\begin{enumerate}
\item  $-1<\zeta(t)<\lambda(t)<1$, 
\vskip .1cm
\item $F=F(\lambda(t),\zeta(t),\zeta(t))$ maps $\Ha$ onto $L(a_0,0,q_0-t)$,
\vskip .1cm
\item  $F(\lambda(t),\zeta(t),\zeta(t)-r(t))$ maps $\Ha$ onto $L(a(t),b(t),q_0)$.
\vskip .1cm
\end{enumerate}
\noindent
Moreover, for small $t$ we have the estimate
\begin{equation}\label{r-t}
r(t)=\frac{(1+o(1)) t}{C_1 \log t^{-1}}.
\end{equation}
for some constant $C_1>0$, and where $o(1) \to 0$ when $t \to 0$. 
\end{lemma}

\subsection{The second path}

We have

\begin{definition} We define $\sigma_2:[0,q_0)\to \T(S_0)$ by letting $\sigma_2(t)=S(a(t),b(t),q_0)$, where $a(t)$, $b(t)$, and $r(t)$, are the functions from Lemma \ref{lemma-novosti-0}.
\end{definition}

Let  
\begin{equation}\label{eq-Gt}
G_t=F(\lambda(t),\zeta(t),\zeta(t)-r(t)) \circ F^{-1}(\lambda(t),\zeta(t),\zeta(t)).
\end{equation}
Then $G_t:L(a_0,0,q_0-t) \to L(a(t),b(t),q_0)$ is a conformal mapping  which maps the vertices  
$P_1,P_2,P_3,P_4,P_5$ on  $L(a_0,0,q_0-t)$ to the respective point on $ L(a(t),b(t),q_0)$.
 
This implies  that the unmarked 3-punctured discs $S(a_0,0,q_0-t)$ and $S(a(t),b(t),q_0)$ are biholomorphic to each other. The biholomorphism between them is obtained by doubling the map $G_t$ (the double of $G_t$ is also denoted by $G_t$). However, $G_t$ does not respect the markings on the boundary of the two 3-punctured discs and therefore $S(a_0,0,q_0-t)\ne S(a(t),b(t),q_0)$ in $\T(S_0)$.

\begin{remark} 
If we glue the two copies of the edges $\overline{P_5P_1}$, and turn $S(a(t),b(t),q_0)$ and $S(a_0,0,q_0-t)$ into 5-punctured spheres (as it was done in \cite{Markovic}), then $S(a(t),b(t),q_0)$ and $S(a_0,0,q_0-t)$ represent the same point in $\T^0_{0,5}$.
\end{remark}

\subsection{The key properties of the two paths} While in the present case the paths $\sigma_1$ and $\sigma_2$ are not the same, we can estimate the Teichm\"uller distance between them. The following lemma is proved in the next section.

\begin{lemma}\label{lemma-12} There exists a  constant $C>0$ such that 
\begin{equation}\label{eq-12}
\left|\Fol(\sigma_1(t))-\Fol(\sigma_2(t))\right|\le C\frac{t^2}{\log^2 t^{-1}},
\end{equation}
where $\Fol:\T(S_0)\to \Ha$ is the holomorphic map from Lemma \ref{lemma-texnika}.  
\end{lemma}

Since $\sigma_2(t)$ lies in the locus $S(a,b,q_0)$, by Lemma \ref{lemma-texnika}   one can compute the value of $\Fol$ on $\sigma_2(t)$. The following lemma is a restatement Lemma 8.2  in \cite{Markovic}.
\begin{lemma}\label{lemma-sigma2} Suppose There are constants $\beta_1$ and $\beta_2\neq 0$ such that
$$
\Fol(\sigma_2(t))=\Fol(\sigma_2(0))+\beta_1(1+o(1))\frac{t}{\log t^{-1}}+\beta_2(1+o(1))\frac{t^2}{\log t^{-1}}+o\left(\frac{t^2}{\log t^{-1}}\right).
$$
\end{lemma}

\subsection{Proof of Theorem  \ref{theorem: 3-punctured disc case}} We prove the theorem assuming Lemma \ref{lemma-12} and Lemma \ref{lemma-sigma2}. The proof is by contradiction. Suppose   $d_C=d_\T$ on $\tau^{\psi(a_0,b_0,q_0)}$. Then there exists  a holomorphic function   $\Fol :\Teich(S_0) \to \Ha$ satisfying the equality from  Lemma \ref{lemma-sigma2}.  

On the other hand, from Lemma \ref{lemma-12} we have that 
$$
\left|\Fol(\sigma_1(t))-\Fol(\sigma_2(t))\right|= O\left(\frac{t^2}{\log^2 t^{-1}}\right)=o\left(\frac{t^2}{\log t^{-1}}\right).
$$
Replacing this into  the equality from  Lemma \ref{lemma-sigma2} we get 
$$
\Fol(\sigma_1(t))=\Fol(\sigma_1(0))+\beta_1(1+o(1))\frac{t}{\log t^{-1}}+\beta_2(1+o(1))\frac{t^2}{\log t^{-1}}+o\left(\frac{t^2}{\log t^{-1}}\right),
$$
where $\beta_2\ne 0$. By Lemma \ref{lemma-kraj-1} the path $\sigma_1$ is infinitely differentiable. But this implies that $\beta_1=\beta_2=0$  which contradicts the fact that $\beta_2\ne 0$. This contradiction proves the theorem.

\section{Proof of Lemma \ref{lemma-12}}

Recall the conformal map $G_t:S(a_0,0,q_0-t) \to S(a(t),b(t),q_0)$ from (\ref{eq-Gt}). Let $g_t$ be the restriction of $G_t$ to the boundary of $S(a_0,0,q_0-t)$ which is  the union of  two copies of $\overline{P_5P_1}$ with their endpoints identified. 
We equip the surface  $S(a(t),b(t),q_0)$ with the non-standard marking $S(a_0,b_0,q_0) \to S(a(t),b(t),q_0)$ which differs from the standard marking in that it is equal to $g_t$ on  the boundary of $S(a_0,0,q_0-t)$.

\begin{definition} We let  $\wh{S}(a(t),b(t),q_0) \in \T(S_0)$ denote the Riemann surface $S(a(t),b(t),q_0)$ with its non-standard marking. 
\end{definition}

By construction, we have that 
\begin{equation}\label{eq-jed}
\sigma_1(t)=\wh{S}(a(t),b(t),q_0),\quad \forall t\in [0,q_0),
\end{equation}
in $\T(S_0)$. Thus, to compare the value of $\Fol$ at $\sigma_1(t)$ and $\sigma_2(t)$ is the same as comparing the value of $\Fol$ at $S(a(t),b(t),q_0)$ and $\wh{S}(a(t),b(t),q_0)$, which represent the same 3-punctured disc with the standard and the non-standard markings respectively.

\subsection{The derivative of the boundary map} After identifying the interval $\overline{P_5P_1}$ with the interval $[0,1]$ we obtain the maps $g_t:[0,1]\rightarrow[0,1]$. With this notation we have the following lemma.

\begin{lemma}\label{lemma-bm} The estimates 
$$
g_t(x)=x+O(r(t)),\quad\text{and}\quad \frac{dg_t}{dx}(x)=1+O(r(t)),
$$
hold   for every  $x\in [0,1]$, where $r(t)$ is the function from Lemma \ref{lemma-novosti-0} which satisfies the estimate (\ref{r-t}).
\end{lemma}
\begin{proof} 
Let $f_1,f_2:[1,+\infty)\rightarrow[0,1)$ denote  the restrictions of  $F(\lambda(t),\zeta(t),\zeta(t)-r(t))$ and  $F(\lambda(t),\zeta(t),\zeta(t))$ to the interval $[1,\infty)$ respectively. Let $0<t_0<q_0$. Since $-1<\zeta(t)<\lambda(t)<1$, and $\zeta(t),\lambda(t)$ are continuous, there exists a constant $\delta_0>0$ such that 
\begin{equation}\label{eq-delta0}
-1+\delta_0\le\zeta(t)<\lambda(t)\le1-\delta_0, \quad\quad \forall t\in [0,t_0].
\end{equation}

To ease the notation  we let $\lambda=\lambda(t)$, $\zeta=\zeta(t)$, and $r=r(t)$.
Then
\begin{align*}
&f_1(x)=J(\lambda,\zeta,0)^{-1}\int_1^x\frac{1}{\sqrt{y+1}\sqrt{y-\lambda}\sqrt{y-1}}dy,\\
&f_2(x)=J(\lambda,\zeta,r)^{-1}\int_1^x\frac{\sqrt{y-\zeta+r}}{\sqrt{y+1}\sqrt{y-\zeta}\sqrt{y-\lambda}\sqrt{y-1}}dy.
\end{align*}
We have $g_t=f_2\circ f_1^{-1}$. By the chain rule we have
\[\frac{dg_t}{dx}(x)=\frac{f_2'(y)}{f_1'(y)}=\frac{J(\lambda,\zeta,0)}{J(\lambda,\zeta,r)}\frac{\sqrt{y-\lambda+r}}{\sqrt{y-\lambda}},\]
where $y=f_1^{-1}(x)$.

From the definition of $J$ (see Section \ref{subsection-SC} above), and using (\ref{eq-delta0}), one easily obtains the estimate 
$$
\frac{J(\lambda,\zeta,0)}{J(\lambda,\zeta,r)}=1+O(r),
$$
when $t\in [0,t_0]$. Again using (\ref{eq-delta0}), and since $y\ge 1$,  we obtain 
$$
\frac{\sqrt{y-\lambda+r}}{\sqrt{y-\lambda}}=1+O(r).
$$ 
Combining the two previous estimates yields
$$
\frac{d}{dx}g_t(x)=1+O(r). 
$$
Integrating this equality gives $g_t(x)=x+O(r)$. The proof is complete.
\end{proof}

\subsection{An auxiliary quasiconformal map}

In this subsection we construct a quasi-conformal homeomorphism $$
f=f_t:S(a(t),b(t),q_0)\to S(a(t),b(t),q_0)
$$
realising the non-standard marking of the surface $S(a(t),b(t),q_0)$ which we denoted by $\wh{S}(a(t),b(t),q_0)$. 
Therefore, $f$  maps the vertices to vertices, and its  restriction to the boundary is given by $g_t$. 

Recall that $S(a(t),b(t),q_0)$ is partitioned into annuli $\Pi_1=\Pi_1(t)$ and $\Pi_2=\Pi_2(t)$ (see Figure \ref{fig:Ls1} above). The annulus $\Pi_2$ is  is obtained by gluing together two copies  of the rectangle $[0,1]\times[0,a(t)]$ along the vertical edges. Recall that $a(t)$ is a positive continuous function for $t\in [0,q_0)$. If we restrict to $[0,t_0]$ for some $t_0<q_0$, we find that there exists a constant $a>0$ such that $0<a<a(t)$ for every $t\in[0,t_0]$. 

Therefore each  $\Pi_2(t)$, where  $t\in[0,t_0]$, contains the annulus $P$ which is the double of the rectangle $[0,1]\times[0,a]$.
We define $f$ to be equal to  the identity map away from $P$. On $P$, the map $f$ is defined as 
 \begin{equation}\label{fa}
 f(x+yi)=\frac{y}{a}x+\left(1-\frac{y}{a}\right)g_t(x)+yi,
 \end{equation}
 where $(x,y)\in [0,1]\times[0,a]$.
 
\begin{lemma}\label{lemma-faa}
Let $\mu_t$ denote the Beltrami dilatation of $f_t$. Then 
\begin{equation}\label{eq-mu}
||\mu_t||_\infty=O\big(r(t)\big),
\end{equation}
and
\begin{equation}\label{eq-mu-1}
\int\limits_{S(a(t),b(t),q_{0})} \mu_t \psi(a(t),b(t),q_0)=o\big(r(t)\big).
\end{equation}
\end{lemma}
\begin{proof}
We have
\begin{align}\label{eq10}
\begin{split}
    &f_z=1+\frac{1}{2}\left(1-\frac{y}{a}\right)\left(\frac{dg_t}{dx}(x)-1\right)+\frac{i}{2a}(g_t(x)-x),\\
    &f_{\bar{z}}=\frac{1}{2}\left(1-\frac{y}{a}\right)\left(\frac{dg_t}{dx}(x)-1\right)-\frac{i}{2a}(g_t(x)-x)
\end{split}
\end{align}
on the first copy of $[0,1]\times[0,a]$, and 
\begin{align}\label{eq11}
\begin{split}
    &f_z=1+\frac{1}{2}\left(1-\frac{y}{a}\right)\left(\frac{dg_t}{dx}(x)-1\right)-\frac{i}{2a}(g_t(x)-x),\\
    &f_{\bar{z}}=\frac{1}{2}\left(1-\frac{y}{a}\right)\left(\frac{dg_t}{dx}(x)-1\right)+\frac{i}{2a}(g_t(x)-x)
\end{split}
\end{align}
on the other copy of $[0,1]\times[0,a]$. From Lemma \ref{lemma-bm} we find
\[f_z=1+O(r(t)),\quad f_{\bar{z}}=O(r(t)),\]
which proves (\ref{eq-mu}).

Let $\mu^1_t$ be the value of $\mu_t$ on the first copy of  $[0,1]\times[0,a]$, and $\mu^2_t$ be the value of $\mu_t$ on the second  copy of  $[0,1]\times[0,a]$. Then
\begin{align*}
\int\limits_{S(a(t),b(t),q_{0})} \mu_t \psi(a(t),b(t),q_0)&=\int\limits_{[0,1]\times[0,a]} (\mu^1_t+\mu^2_t)\, dz^2 \\
&=\int\limits_{[0,1]\times[0,a]} \frac{1}{2}\left(1-\frac{y}{a}\right)\left(\frac{dg_t}{dx}(x)-1\right)\, dxdy+o(r(t)).
\end{align*}
The last integral is easily seen to be equal to zero using Fubini's theorem. This proves the lemma.
\end{proof}

\subsection{Proof of Lemma \ref{lemma-12}} To ease the notation we let $S_t=S(a(t),b(t),q_0)$. Note that $\T(S_0)\equiv \T(S_t)$. Recall the holomorphic map $\BD_1(S_t)\to \T(S_t)$. Define $\Phi:\BD_1(S_t)\to \Ha$ to be the composition of $\Fol:\T(S_t)\to \Ha$, and the map $\BD_1(S_t)\to \T(S_t)$. Then $\Phi$ is holomorphic. Moreover, from the construction we have
$$
\Phi(0)=\Fol(\sigma_2(t)),\quad \text{and}\quad \Phi(\mu_t)=\Fol(\sigma_1(t)).
$$
From Lemma \ref{lemma-ab} (in particular, from (\ref{eq-ab}))   we find that 
$$
\Phi(0)- \Phi(\mu_t)=-\imu \int\limits_{S_{t}} \mu_t \psi(a(t),b(t),q_0) + O(||\mu_t||^2_\infty).
$$
Thus, from (\ref{eq-mu-1}), and then (\ref{eq-mu-1}), we get
$$
\Phi(0)- \Phi(\mu_t)=O(||\mu_t||^2_\infty)=O(r^2(t)).
$$
Replacing (\ref{r-t}) in the last estimate gives
$$
\Phi(0)- \Phi(\mu_t)=O(r^2(t))=\frac{(1+o(1))^2 t^2}{C_1 \log^2 t^{-1}}=O\left(\frac{t^2}{\log^2 t^{-1}}\right).
$$
The lemma is proved.

\section{The appendix: Linear functionals on the tangent space of $\T(S)$} 
The purpose of the appendix is to prove Proposition \ref{prop-tame}. 
Let $S$ be a topologically finite Riemann surface. Recall that every non-zero holomorphic quadratic differential 
$\varphi\in \QD(S)$ yields the Beltrami differential $\mu_\varphi=\frac{|\varphi|}{\varphi}\in \BD(S)$. 
\begin{lemma}\label{lemma-tame} Suppose that $\varphi \in \QD(S)$ is is a Jenkins-Strebel differential. Then
$$
\lim\limits_{n\to \infty} \int\limits_S \mu_\varphi \psi_n = 0,
$$
for every sequence $\{\psi_n\}_{n\in \N} \subset \QD(S)$  which converges to zero on every compact subset of $S$.
\end{lemma}
\begin{remark} Note that we do not have to assume that the $L^1$-norm of the sequence $\psi_n$ is uniformly bounded in $n$.
\end{remark}
\begin{proof} Let $\Pi_1,...,\Pi_k$ denote the  peripheral annuli from the decomposition of $S$ into annuli which are foliated by closed horizontal trajectories of $\varphi$. Since $\psi_n \to 0$ on compact subset of $S$, to prove the lemma it suffices to prove that for each $\Pi\in \{\Pi_1,...,\Pi_k\}$ we have
\begin{equation}\label{eq-tram}
\lim\limits_{n\to \infty} \int\limits_\Pi \mu_\varphi \psi_n = 0.
\end{equation}
Let $r_0>1$ be such that $\Pi$ is conformally equivalent to the annulus $A=\{r_{0}^{-1}< |z| <r_0\}$.  
Then 
$$
\mu_\varphi(z)=\frac{z^2}{|z|^2}\frac{d\overline{z}}{dz},\quad z\in A.
$$ 
Moreover, we let $f_n$ be the holomorphic functions on $A$ such that $\psi_n=f_ndz^2$. We have
\begin{equation}\label{eq-mol}
\int\limits_\Pi \mu_\varphi \psi_n=\int\limits_{r_{0}^{-1}< |z| <r_0}\frac{z^2}{|z|^2}\, f_n(z) \, dxdy. 
\end{equation}
Let $g_n(z)=z^2f_n(z)$, and set
$$
I_n=\int\limits_{0}^{2\pi} g_n(re^{i\theta})\, d\theta.
$$
Since $g_n$ is holomorphic the integral $I_n$ does not depend on $r$. Moreover, since $\psi_n\to 0$ on compact subsets of $S$ it follows that $g_n\to 0$ on compacts subsets of $A$. Thus, 
\begin{equation}\label{eq-mol-1}
\lim_{n\to \infty} I_n=0.
\end{equation}
On the other hand, passing to polar coordinates we get
$$
\int\limits_{r_{0}^{-1}< |z| <r_0}\frac{z^2}{|z|^2}\, f_n(z) \, dxdy=\int\limits_{r^{-1}_{0}}^{r_0}\int\limits_{0}^{2\pi} g_n(re^{i\theta}) \frac{1}{r}\, d\theta dr=
I_n\int\limits_{r^{-1}_{0}}^{r_0} \frac{1}{r}\, dr =2 I_n \log r_0.
$$
Applying (\ref{eq-mol-1}) and (\ref{eq-mol}) proves the lemma.
\end{proof}

\begin{proposition}\label{prop-tame} Let $L:T_S\T(S)\to \C$ be a linear functional of norm one. Suppose  that $L([\mu_\varphi]^*)=1$ for some Jenkins-Strebel differential $\varphi\in \QD(S)$ of norm one. Then 
\begin{equation}\label{eq-tame}
L([\nu]^*)=\int\limits_S\varphi\nu,
\end{equation}
for all $\nu\in \BD(S)$.
\end{proposition}
\begin{proof} As stated in Section \ref{subsection-well}
we know that $T_S\T(S)$ is isometric to the dual of $\QD(S)$. 
The Banach-Alaoglu theorem implies that for any finite dimensional subspace  $V< \BD(S)$ there exists a sequence $\psi_n\in \QD(S)$, of norm one elements, such that 
$L([\mu]^*)=\lim_{n\to \infty} \int\limits_S \mu \varphi_n$, for every $\mu\in V$.

Fix $\nu \in \BD(S)$. Thus, there exists a sequence $\psi_n\in \QD(S)$, with $||\psi_n||_1=1$, such that 
\begin{equation}\label{eq-tame-1-1}
L([\mu_\varphi]^*)=\lim_{n\to \infty} \int\limits_S \mu_\varphi \psi_n,
\end{equation}
and
\begin{equation}\label{eq-tame-1-2}
L([\nu]^*)=\lim_{n\to \infty} \int\limits_S \nu \psi_n,
\end{equation}
After passing onto a sequence if necessary,   
$\psi_n$ converges uniformly on compact subset of $S$ to some $\psi\in \QD(S)$, with   $||\psi||_1\leq 1$. 

\begin{claim}\label{claim-tame} The equalities $||\psi||_1=1$, and $\int\limits_S \mu_\varphi \psi=1$, hold. 
\end{claim}
\begin{proof} The sequence $\psi-\psi_n$ converges to zero on compact subset of $S$. From Lemma \ref{lemma-tame}  we conclude that  $\int\limits_S \mu_\varphi (\psi-\psi_n)\to 0$, $n\to \infty$. Combining this with (\ref{eq-tame-1-1}), and the fact that  $L([\mu_\varphi]^*)=1$, implies  $\int\limits_S \mu_\varphi \psi=1$. Combining this equality  with  
the fact $||\psi||_1\leq 1$ shows that we must have $||\psi||_1=1$.
 \end{proof} 
Combining the two equalities from Claim \ref{claim-tame}, we obtain
$$
1=\text{Re}\int\limits_S \mu_\varphi \psi\le \int\limits_S |\psi|=1.
$$
The conclusion is that  $\psi$ is a positive scalar multiple of $\varphi$ almost everywhere on $S$. Since both $\psi$ and $\varphi$ are holomorphic, and $||\psi||_1=||\varphi||_1=1$, the equality $\psi=\varphi$ holds.  

Now, since $\psi_n\to \varphi$, when $n\to \infty$,  on compact subsets of $S$, and since $||\psi_n||_1=||\psi||_1$, for every $n$, it follows that $||\psi_n-\varphi||_1\to 0$, when $n\to \infty$. Combining this with (\ref{eq-tame-1-2})  proves (\ref{eq-tame}).  
\end{proof}

\end{document}